\documentclass[12pt]{amsart}
\usepackage{xspace,amssymb,amsfonts,euscript,eufrak,mathrsfs}

\usepackage{amsthm,amsmath}
\usepackage{url}
\usepackage{graphicx}
 \usepackage{epstopdf}

\usepackage{srcltx}
 
\RequirePackage{color}
\definecolor{myred}{rgb}{0.75,0,0}
\definecolor{mygreen}{rgb}{0,0.5,0}
\definecolor{myblue}{rgb}{0,0,0.65}

\RequirePackage{ifpdf}
\ifpdf
  \IfFileExists{pdfsync.sty}{\RequirePackage{pdfsync}}{}
  \RequirePackage[pdftex,
   colorlinks = true,
   urlcolor = myblue, 
   citecolor = mygreen, 
   linkcolor = myred, 
   pagebackref,
   bookmarksopen=true]{hyperref}
\else
  \RequirePackage[hypertex]{hyperref}
\fi

\title[Parity sheaves and moment graphs]{Parity sheaves, moment graphs and the 
$p$-smooth locus of Schubert varieties}

\author{Peter Fiebig}

\author{Geordie Williamson}

\input xy
\xyoption {all}

\newcommand{\CB}{{\mathcal B}}
\newcommand{\CC}{{\mathcal C}}

\newcommand{\CE}{{\mathcal E}}
\newcommand{\CF}{{\mathcal F}}
\newcommand{\CG}{{\mathcal G}}

\newcommand{\CI}{{\mathcal I}}
\newcommand{\CJ}{{\mathcal J}}
\newcommand{\CK}{{\mathcal K}}

\newcommand{\CP}{{\mathcal P}}
\newcommand{\CQ}{{\mathcal Q}}
\newcommand{\CR}{{\mathcal R}}

\newcommand{\CV}{{\mathcal V}}

\newcommand{\CZ}{{\mathcal Z}}

\newcommand{\SB}{{\mathscr B}}

\newcommand{\SF}{{\mathscr F}}

\newcommand{\SM}{{\mathscr M}}
\newcommand{\SN}{{\mathscr N}}

\newcommand{\fh}{{{\mathfrak h}}}

\newcommand{\fg}{{{\mathfrak g}}}

\newcommand{\fn}{{{\mathfrak n}}}

\newcommand{\DC}{{\mathbb C}}
\newcommand{\DP}{{\mathbb P}}

\newcommand{\DZ}{{\mathbb Z}}

\newcommand{\DQ}{{\mathbb Q}}

\newcommand{\DF}{{\mathbb F}}
\newcommand{\DD}{{\mathbb D}}
\newcommand{\DW}{{\mathbb W}}

\newcommand{\Hom}{{\operatorname{Hom}}}

\newcommand{\id}{{\operatorname{id}}}
\newcommand{\res}{{\operatorname{res}}}

\newcommand{\Gr}{{\operatorname{{\mathcal{G}{r}}}}}
\newcommand{\Fl}{{\operatorname{{\mathcal{F}{l}}}}}

\newcommand{\catmod}{{\operatorname{-mod}}}

\newcommand{\inj}{{\hookrightarrow}}

\newcommand{\ol}{\overline}

\newcommand{\Stab}{{\operatorname{Stab}}}

\newcommand{\pt}{{pt}}

\newcommand{\linie}{{\,\text{---\!\!\!---}\,}}
\newcommand{\llinie}{{\text{---\!\!\!---\!\!\!---}}}

\newcommand{\For}{{\operatorname{For}}}

\newcommand{\THypgr}{{\operatorname{\mathbb{H}_T^\bullet}}}
\newcommand{\GHypgr}{{\operatorname{\mathbb{H}_G^\bullet}}}

\newcommand{\THyp}{\operatorname{\mathbb{H}}_T}

\newcommand{\GHyp}{\operatorname{\mathbb{H}}_G}

\newcommand{\ul}{\underline}

\newcommand{\comment}[1]{}

  \newtheorem{theorem}{Theorem}[section]
  \newtheorem{definition}[theorem]{Definition}
  \newtheorem{lemma}[theorem]{Lemma}
  
  \newtheorem{proposition}[theorem]{Proposition}
  \newtheorem{corollary}[theorem]{Corollary}

  \theoremstyle{remark}
  \newtheorem{remark}{Remark}

\newcommand{\excise}[1]{}
  
  \def\varle{\leqslant}

\newcommand{\IC}{\mathbf{IC}}
\DeclareMathOperator{\codim}{codim}
\DeclareMathOperator{\Rep}{Rep}

\begin{document}
\begin{abstract}
We show that, with coefficients in a field or complete local
principal ideal domain $k$, the Braden-MacPherson algorithm computes
the stalks of parity sheaves with coefficients in $k$. As a
consequence we deduce that the Braden-MacPherson algorithm may be used
to calculate the characters of tilting modules for algebraic groups
and show that the $p$-smooth locus of a (Kac-Moody) Schubert variety
coincides with the rationally smooth locus, if the underlying Bruhat
graph satisfies a GKM-condition.
\end{abstract}
\maketitle
\tableofcontents

\section{Introduction}

In Lie theory, one of the most successful methods to calculate
representation theoretic data (such as character formulae, decomposition
numbers or Jordan--H\"older multiplicities) is to find a geometric or topological
interpretation of the problem. In many  examples one obtains representation theoretic information 
from the stalks of intersection cohomology complexes on an associated algebraic variety (for example the flag variety,
the nilpotent cone, or the group itself).

In the most successful applications of this approach (the
Kazhdan-Lusztig conjecture, canonical bases, character sheaves \dots ) 
the representation theoretic objects under consideration are assumed
to be defined over a field of characteristic $0$. In this case the
decomposition theorem often allows one to recursively calculate the
stalks of intersection cohomology complexes, hence solving (or at
least providing a combinatorial algorithm to solve) the representation theoretic problem.

However, recently a
number of authors have pointed out that geometry also has something to
say about modular representation theory (see \cite{JMW1} for a
survey). In this article we are motivated by the following two
examples of this phenomenon:
\begin{itemize}
\item For a ring $k$, the geometric Satake equivalence (cf.  \cite{MV})  provides an equivalence of
 tensor categories between equivariant perverse $k$-sheaves on the
  affine Grassmannian and rational representations of the
 Langlands dual group scheme over $k$.
\item In \cite{Fie2} and
  \cite{Fie3} a certain category of sheaves of $k$-vector spaces on an affine flag variety was related to representations of the $k$-Lie algebra associated to the Langlands dual root system. Here $k$ is assume to be a field whose  characteristic is required to be at least the associated Coxeter number. The relation is then used to give a new proof of Lusztig's
  conjecture on the simple rational characters for reductive groups in almost all characteristics.
\end{itemize}

In \cite{JMW2} (motivated by ideas of Soergel \cite{Soe} together
with a desire to better understand such
relationships) a new class of sheaves, the ``parity sheaves'', was introduced. These
are certain constructible sheaves on a stratified algebraic variety,
which satisfy a parity vanishing condition with respect to stalks and
costalks. It was shown that, under some additional assumptions, the
indecomposable parity sheaves are parametrized in the same way as the
intersection cohomology complexes. If the coefficients are of
characteristic $0$ the decomposition theorem often implies that the
indecomposable parity sheaves are isomorphic to intersection
cohomology complexes (up to a shift). 

In positive characteristics this
needs no longer be true.  However, with coefficients of positive
characteristic parity sheaves are often easier to work with than
intersection cohomology complexes. Moreover, for some representation theoretic applications they may even form their natural
replacement. For example,
\begin{itemize}
\item the category considered in \cite{Fie2, Fie3} turns out to be the
  category of parity sheaves,
\item under the geometric Satake equivalence (and under some mild and
  explicit assumptions on the characteristic of $k$) the parity
sheaves correspond to tilting modules for the Langlands dual group (cf. \cite{JMW3}).
\end{itemize}
In the above results, fundamental representation theoretic data is
encoded in the stalks of the indecomposable parity sheaves. It is
therefore an important problem to find an algorithm for their calculation.

For an appropriately stratified complex algebraic variety $X$ with torus action Braden and
MacPherson \cite{BrM} describe an algorithm for calculating the stalks
of intersection cohomology complexes with coefficients in a field of
characteristic $0$ (using localisation techniques in equivariant
cohomology developed by Goresky, Kottwitz and MacPherson \cite{GKM}).  The torus action has, by assumption, only finitely many fixed points and
one-dimensional orbits. The structure of the one-skeleton of the
torus action can be
encoded in the ``moment graph'' of the variety:
\begin{itemize}
\item the vertices and edges are given by the torus fixed points and
  one-dimensional orbits respectively, with a one-dimensional orbit
  incident to those fixed points in its closure,
\item each edge is labelled by a character of the torus
  determining an isomorphism of the orbit with $\mathbb{C}^*$ (this is defined only up to a sign).
\end{itemize}

Braden and MacPherson then describe an algorithm (using only
commutative algebra) to produce a ``sheaf'' on the
moment graph, and show that its stalks agree with those of the equivariant
intersection cohomology complex. Thus the (a priori extremely
difficult) computation of the stalks of the intersection cohomology
complex may (in principle) be carried out in an elementary way.

Now, the Braden--MacPherson algorithm makes sense with
coefficients in an arbitrary field $k$, or even in a local ring. However, simple examples show
that it does not compute
the stalks of intersection cohomology complexes when the
coefficients are not of characteristic $0$. The central result of
this paper is the following:

\begin{theorem}  Suppose that the pair $(X,k)$ satisfies  the GKM-condition (cf.~ Section \ref{subsec-locpartII}). 
Then the Braden-MacPherson algorithm computes the stalks of indecomposable parity sheaves.    
\end{theorem}

In the theorem, $k$ denotes a complete local principal ideal
domain. If $k$ is a field, then the GKM-condition may be stated
simply: one requires that, for all pairs of one-dimensional orbits
having a common torus fixed point in their closure, the corresponding  characters
 do not become linearly dependent modulo
$k$. This condition can easily be read off the associated
moment graph.

 In the course of the proof of the above result we provide a
 version of localisation theorem of \cite{GKM} with coefficients in a ring, i.e. we show that the hypercohomologies of certain equivariant sheaves on $X$ are given by the global sections of  associated moment graph sheaves (see Theorem \ref{theorem-loc2}). For
complete local rings we then show that the Braden--MacPherson algorithm yields the objects associated to parity sheaves (see Theorem \ref{thm:Wparity=BM}). In contrast to the proof of Braden--MacPherson, our arguments are more elementary, as we do not need the theory of mixed Hodge modules. As in  characteristic $0$ the decomposition theorem implies that the parity sheaves are intersection cohomology complexes up to a shift, we obtain a new proof of their result.

Applying the above theorem to
the affine Grassmannian and using the Satake equivalence, we obtain:

\begin{theorem} Suppose that $p>h+1$, where $h$ denotes the Coxeter number of our datum. On the moment graph of the affine Grassmannian and  with coefficients in the ring of $p$-adic integers, the Braden--MacPherson algorithm calculates the characters of tilting modules of the Langlands dual group over $\ol\DF_p$.
\end{theorem}

The moment graph of the affine Grassmannian is GKM for a
field $k$ if and only if $k$ is of characteristic $0$. We avoid this complication by using the $p$-adic integers in the above theorem.

We apply the  multiplicity one result of \cite{Fie06} to obtain a description of the $p$-smooth locus of Schubert varieties. Recall that an $n$-dimensional algebraic  variety $X$ is $p$-smooth if for all $x\in X$ one has an isomorphism of graded vector spaces
$$
H^\bullet(X,X\setminus\{x\},\DF_p)\cong H^\bullet(\DC^n,\DC^n\setminus\{0\},\DF_p).
$$
The $p$-smooth locus of $X$ is the largest open $p$-smooth
subvariety. One similarly defines rationally smooth, and the rationally
smooth locus by replacing $\DF_p$ by $\DQ$ above. 
If $X$ is rationally (resp. $p$-) smooth it satisfies Poincar\'e duality with
rational (resp. $\DF_p$-) coefficients. Here is our result:
 
\begin{theorem} \label{thm:intropsmooth}
Let $\CG$ be the  moment graph of a (Kac-Moody) Schubert variety $X$ and suppose that $(\CG,\DF_p)$ is a  
  GKM-pair. Then the $p$-smooth locus of $X$ coincides with its
  rationally smooth locus.
\end{theorem}

In the finite dimensional case, the GKM-condition is always satisfied if $p\ne 2$ and if, in addition, $p\ne 3$ in  $G_2$.
This answers a (stronger version of) a question of Soergel (cf. 
\cite{Soe}). In fact, we prove the above theorem for a larger class of
varieties with an appropriate torus action for fields $k$ that satisfy the GKM-condition.

The smooth and rationally smooth locus of Schubert varieties has been
the subject of much investigation by a number of authors. See for example
\cite{Carrell1}, \cite{Kumar1}, \cite{Dyer1}, \cite{Dyer2} and \cite{Ar3}. 
It a result known as Peterson's theorem that the smooth and rationally smooth locus
agree in simply-laced type, which immediately implies the above
theorem. However, there are examples in non-simply-laced
types of small rank where the $2$-smooth and $3$-smooth locus
do not agree with the rationally smooth locus.

Lastly let us remark that results of this paper (in particular
Section \ref{subsec:freeness}) are used in an
essential way in the paper \cite{JW}, which shows that that Kumar's
criterion for the rational smoothness of Schubert varieties can be
extended to provide a criterion for $p$-smoothness. In particular, the
main result of \cite{JW} provides a means to determine the $p$-smooth locus when the
underlying moment graph is not GKM, complementing Theorem
\ref{thm:intropsmooth}. On may also use Theorem
\ref{thm:intropsmooth} together with the results of \cite{JW} to
obtain a novel proof of Peterson's theorem.

\subsection{Acknowledgements} We would like to thank Daniel Juteau and
Olaf Schn\"urer for useful conversations and Michel Brion for useful
correspondence. P.F. gratefully acknowledges the support of the Landesstiftung Baden--W\"urttemberg as well as the DFG-Schwerpunkt 1388 `Representation Theory''. Both authors gratefully acknowledge the support of the Newton Institute in Cambridge, where parts of this paper were written.

\section{Equivariant sheaves}
In this section we recall the construction  of the bounded equivariant derived
category $D^b_G(X,k)$ that is associated to a topological group $G$, a
ring of coefficients $k$ and a $G$-space $X$.
To a suitable continuous
$G$-equivariant map $f\colon X\to Y$ one associates the push-forward functors
$$
f_\ast, f_!\colon D^b_G(X,k)\to D^b_G(Y,k)
$$
and the pull-back functors
$$
f^\ast,f^!\colon D^b_G(Y,k)\to D^b_G(X,k)
$$
satisfying a Grothendieck formalism. We then recall the equivariant cohomology  $\GHypgr(\CF)$ of $X$ with coefficients in  $\CF\in D^b_G(X,k)$ and, finally, the Mayer--Vietoris sequence
associated to an open $G$-stable covering $X=U\cup V$.

We will be mainly concerned with the following
situation: $G$ will either be a complex algebraic torus,
i.e.~ $G\cong(\DC^\times)^r$ for some $r>0$, endowed with its metric
topology, or its compact subtorus $(S^1)^r$. The
space $X$ will be a complex algebraic variety with an algebraic
$G$-action, and endowed with its metric topology. The main reference
for the following is \cite{BLu}.

\subsection{The equivariant derived category of a $G$-space}
We fix a topological group $G$. A $G$-space is a topological space
endowed with a continuous $G$-action. There always exists a
contractible $G$-space with a topologically free $G$-action. We fix
one of those and call it $EG$. For any $G$-space $X$ we can now define
the quotient $X_G:=X\times_G EG$ of $X\times EG$ by the diagonal
$G$-action. Then we have two maps

\centerline{
\xymatrix{
& X\times EG \ar[ld]_p\ar[rd]^q&\\
X&&X_G.
}
}
 \noindent
The map $q$ on the right is the canonical quotient map and $p$ is the
projection onto the first factor.

Now we fix a ring of coefficients $k$. For any topological space $Y$ 
we denote by $D(Y,k)$ the derived category of sheaves of $k$-modules
on $Y$. By $D^b(Y,k)$ we denote the full subcategory of objects with
bounded cohomology. For a continuous map $f\colon Y\to Y^\prime$ we
then have the push-forward functor $f_\ast\colon D(Y,k)\to
D(Y^\prime,k)$ and the pull-back functor $f^\ast\colon
D(Y^\prime,k)\to D(Y,k)$ (see \cite{Spalt}).

\begin{definition} The {\em equivariant derived category} of sheaves on $X$ with coefficients in $k$ is the full subcategory $D_G(X,k)$ of $D(X_G,k)$ that contains all sheaves $\CF$ for which there is a sheaf $\CF_X\in D(X,k)$ such that $q^\ast\CF\cong p^\ast\CF_X$.
 \end{definition}
We denote by $D^b_G(X,k)\subset D_G(X,k)$ the full subcategory of
objects with bounded cohomology, i.e.~ of
objects that are contained in
$D^b(X_G,k)$.

It turns out that the categories $D_G(X,k)$ and $D_G^b(X,k)$ are independent of the choice of $EG$.
Since $p$ is a trivial fibration with contractible fibre $EG$, the functor $p^\ast\colon D(X,k)\to D(X\times EG,k)$ is a full embedding.
We deduce that for $\CF\in D_G(X,k)$ the sheaf $\CF_X\in D(X,k)$ appearing in the definition above is unique up to unique isomorphism, so  the map $\CF\mapsto \CF_X$ even extends to a functor $\For\colon D_G(X,k) \to D(X,k)$.

\subsection{The equivariant functor formalism}

In order to ensure that all the functors that we introduce in the
following exist we assume that $X$ is a complex algebraic variety
endowed with its metric topology, and that $G$ is a Lie group acting
continuously on $X$.

If $f\colon X\to Y$ is a continuous $G$-equivariant map 
then we get an induced map
$f_G:=f\times_G \id\colon X_G\to Y_G$ and corresponding functors
$f_G^\ast$, $f_{G\ast}$, $f_G^!$ and $f_{G!}$ between the categories
$D^b(X_G,k)$ and $D^b(Y_G,k)$. (Some care is needed in the definition
of $f_G^!$ and $f_{G!}$ because $X_G$ and $Y_G$ are not locally
compact in general. In \cite{BLu} this problem is overcome by
considering $X_G$ as a direct limit of locally compact spaces.
It is also possible to prove the existence and basic properties of
$f_G^!$ in a relative setting, see \cite{SHS}.) It turns out that all
four functors induce functors between the subcategories $D^b_G(X,k)$
and $D^b_G(Y,k)$. By abuse of notation we denote these functors by
the symbols $f^\ast$, $f_\ast$, $f^!$ and $f_!$.

 For a $G$-stable subvariety $i\colon Y \inj X$ and a sheaf $\CF\in D_G^b(X,k)$ we define
 $$
 \CF_Y:=i^\ast \CF.
 $$
So $\CF_Y$ is an object in $D^b_G(Y,k)$.

\subsection{Equivariant cohomology}
The {\em equivariant cohomology $H_G^\bullet(X,k)$ of  $X$ with coefficients in $k$} is the (ordinary) cohomology of the space $X_G$, i.e.
$$
H_G^\bullet(X,k):=H^\bullet(X_G,k).
$$
In particular, the equivariant cohomology of a point is the cohomology
of the {\em classifying space}
$$
BG:=\pt_G=EG\times_G \pt=EG/G
$$
of $G$.

Now let $\CF\in D_G^b(X,k)$. The {\em equivariant cohomology 
$\GHyp^\bullet(\CF)$ of $X$ with coefficients in
$\CF$} is defined as follows. We denote by $\pi\colon X\to \pt$ the map to a
point. Then we have the object $\pi_\ast\CF\in D_G^b(\pt,k)\subset
D^b(BG,k)$ and we define
$$
\GHyp^\bullet(\CF):=H^\bullet(\pi_\ast \CF),
$$
where on the right we have the ordinary cohomology of $BG$ with coefficients in the sheaf
$\pi_\ast \CF$. This is naturally a graded module over $H_G^\bullet(\pt,k)=H^\bullet(BG,k)$,
so equivariant cohomology is a functor
$$
\GHyp^\bullet\colon D_G^b(X,k)\to H_G^\bullet(\pt,k)\catmod^\DZ.
$$
Here and in the following we denote by $A\catmod^\DZ$ the category of $\DZ$-graded modules over a $\DZ$-graded ring $A$. For a graded $A$-module $M=\bigoplus_{n\in\DZ}M_n$ and $l\in\DZ$ we denote by $M[l]$ the graded module obtained by a shift such that $M[l]_n=M_{l+n}$ for all $n\in\DZ$.

Let $i\colon Y\inj X$ be a locally closed $G$-stable subvariety and $\CF\in
D_G^b(X,k)$.
The adjunction morphism
$\id\to i_\ast i^\ast$ yields a morphism $\CF\to i_\ast i^\ast
\CF=i_\ast \CF_Y$. After applying equivariant cohomology this
yields a homomorphism
$$
\GHyp^\bullet(\CF)\to \GHyp^\bullet(i_\ast\CF_Y)=\GHyp^\bullet(\CF_Y)
$$
of $H_G^\bullet(\pt,k)$-modules. We call such a homomorphism a  {\em restriction homomorphism}.

\subsection{The Mayer--Vietoris sequence}
We will often make use of the equivariant Mayer--Vietoris
sequence. Note that the equivariant statement is a straightforward consequence of the
non-equivariant one (see, for example, \cite[2.6.28]{KS1}).

\begin{proposition} \label{prop:MaVe}
Let $X = U \cup V$ where $U, V \subset X$
are open and $G$-stable. Then, given any $\mathcal{F} \in D_G^b(X,k)$,
we have a long exact sequence of equivariant cohomology
\begin{align*}
\dots \to
\mathbb{H}_{G}^{j-1}(\mathcal{F}_{U \cap V}) \to
\mathbb{H}_{G}^j(\mathcal{F}) \to
& \mathbb{H}_{G}^j(\mathcal{F}_U) \oplus
\mathbb{H}_{G}^j(\mathcal{F}_V) \to \\ & \to
 \mathbb{H}_{G}^j(\mathcal{F}_{U \cap V}) \to
\mathbb{H}_{G}^{j+1}(\mathcal{F}) \to \dots.
\end{align*}
\end{proposition}

\subsection{The case of a torus} \label{sec:torus}
Let us suppose now that $G=T$ is a complex torus, i.e.~ a topological
group isomorphic to $(\DC^\times)^r$ for some $r>0$, endowed with the
metric topology.

For $n\ge 0$ we consider the space $(\DC^n\setminus \{0\})^r$
together with the $T$-action given by
$$
(t_1,\dots,t_r)\cdot(x_1,\dots,x_r)=(t_1\cdot x_1,\dots,t_r\cdot x_r).
$$
The embeddings $\DC^n\setminus \{0\}\to \DC^{n+1}\setminus \{0\}$ that map $(z_1,\dots, z_n)$ to
$(z_1,\dots,z_n,0)$ define a direct system
$$
\dots\to (\DC^n\setminus \{0\})^r\to(\DC^{n+1}\setminus \{0\})^r\to \dots
$$
of $T$-spaces. The direct limit
$(\DC^\infty\setminus \{0\})^r:=\lim (\DC^n\setminus \{0\})^r$ is a contractible space together
with a topologically free $T$-action, hence can be chosen as a model
for $ET$.

We denote by $X^\ast(T)$ the character lattice
$\Hom(T,\DC^\times)$ of $T$. Let 
$$
S_k:=S(X^\ast(T)\otimes_\DZ k)
$$
be the symmetric algebra over the free $k$-module $X^\ast(T)\otimes_\DZ k$, graded in such a way that $X^\ast(T)\otimes_\DZ k\subset S_k$ is the homogeneous component of degree $2$. Then the Borel homomorphism (cf.~ \cite{B}, \cite{JanNotes}) gives
a canonical identification
$$
S_k\stackrel{\sim}\to H^\bullet(BT,k)=H_T^\bullet(\pt,k).
$$

\subsection{An attractive proposition}

Now let $X$ be a complex $T$-variety. Recall that a $T$-fixed point $x \in X$ is called \emph{attractive} if all weights of $T$ on the tangent space to $X$ at $x$ lie in an open half space of $X^\ast(T) \otimes_{\mathbb{Z}} \mathbb{R}$. If this is the case then one can find a one parameter subgroup $\alpha: \mathbb{C}^ \times \to T$ and an open neighbourhood $U$ of $x$ such that 
\begin{equation} \label{eq:attract}
\lim_{z \to 0} \alpha(z) \cdot y = x\text{ for all $y \in U$.}
\end{equation}
 If, in addition, $X$ is connected and affine, then $x$ is the unique $T$-fixed point of $X$ and \eqref{eq:attract} holds for all $y \in X$. In particular, the smallest $T$-stable open neighbourhood of $x \in X$ is $X$ itself.

Suppose for the remainder of this section that $X$ is connected and affine, and that $x \in X$ is an attractive fixed point. We denote by $i \colon \{x\}\to X$ the inclusion and by $\pi \colon X\to \{x\}$ the projection. If we apply the functor $\pi_{\ast}$ to
the natural transformation $\id\to i_{\ast}i^\ast$ we get a
natural transformation $\pi_{\ast}\to \pi_{\ast}i_{\ast}i^\ast$. Since $\pi \circ i$ is the identity, we get a natural morphism
$$
\pi_{\ast}\to i^\ast
$$
of functors from $D_T^b(X,k)$ to $D_T^b(\{x\}, k)$.

The goal of the rest of this section is to prove the following (for
similar statements in the non-equivariant or ``weakly equivariant'' setting see
\cite{SprPurity} and \cite{Br}):

\begin{proposition} \label{prop:attr} Suppose that $X$ is connected and affine and that $x\in X$ is an attractive fixed point. Then the morphism of functors
$\pi_{\ast} \to i^\ast$ is an isomorphism.
\end{proposition}

We begin with some lemmata. Suppose we have a pair of Cartesian squares
\[
\xymatrix{
\widetilde{F} \ar[r]^{\widetilde{i}} \ar[d]^q & \widetilde{X} \ar[r] \ar[d]^q \ar[r]^{\widetilde{\pi}} &
\widetilde{F} \ar[d]^q\\
F \ar[r]^{i} & X \ar[r] \ar[r]^{\pi} & F
} 
\]
such that $q$ is smooth and surjective, and $\pi \circ i = \id$ (and
hence $\widetilde{\pi} \circ \widetilde{i} = \id$). The adjunctions
$(\pi^*, \pi_*)$ and $(\widetilde{\pi}^*, \widetilde{\pi_*})$ give
morphisms of functors
\[
\pi_* \to i^* \quad \text{ and } \quad \widetilde{\pi}_* \to \widetilde{i}^*.
\]

\begin{lemma} Let $\CF \in D^b(X,k)$. Then $\pi_* \CF \to i^* \CF$ is an
  isomorphism if and only if $\widetilde{\pi}_* q^*\CF \to \widetilde{i}^* q^*\CF$ is an
  isomorphism.
\end{lemma}

\begin{proof} Because $q$ is surjective, $\pi_* \CF \to i^* \CF$ is an
  isomorphism if and only if $q^* \pi_* \CF \to q^* i^* \CF$ is an
  isomorphism. Now $q^* i^* \CF \stackrel{\sim}{\to} \widetilde{i}^* q^* \CF$ and $q^*
  \pi_* \CF \stackrel{\sim}{\to} \widetilde{\pi}_* q^* \CF$ by smooth
  base change. Via these canonical isomorphisms we obtain a map
\[
\widetilde{\pi}_* q^*\CF \to \widetilde{i}^* q^*\CF.
\]
This is the same map (up to
isomorphism) as that
coming from the morphism $\widetilde{\pi}_* \to \widetilde{i}^*$ (cf. \cite[Theorem 1.8]{BLu}.)
\end{proof}

Now suppose a torus $T$ contracts a variety $X$ onto a fixed locus $F
\subset X$. Consider the diagram
\[
\xymatrix{
X \ar@<1ex>[d]^{\pi}& \ar[l]_p X \times ET \ar[r]^{q} \ar@<1ex>[d]^{\pi}
& X \times_T ET \ar@<1ex>[d]^{\pi} \\
F \ar@<1ex>[u]^i & \ar[l]_p F \times ET \ar@<1ex>[u]^i \ar[r]^q & F \times_T ET
\ar@<1ex>[u]^i }.
\]
Both $p$ and $q$ are smooth, and so applying the above lemma twice we
see that, given $\CF \in D^b_T(X,k)$, we have that $\pi_* \CF \to i^* \CF$ is
an isomorphism in $D^b_T(F,k)$ if and only if $\pi_* \For(\CF) \to i^*
\For(\CF)$ is.

Given a $G$-space $X$, let us call $\CF \in D^b(X,k)$ \emph{naively
  equivariant} if we have an isomorphism $m^* \CF \to p^* \CF$ where
$m$ and $p$ denote the action and projection maps
\[
\xymatrix{ G \times X \ar@<0.5ex>[r]^m \ar@<-0.5ex>[r]_p & X. }
\]
Note that, if $G$ acts freely on $X$ then pullback along $X \to X/G$
allows us to view any $\CF \in D^b(X/G,k)$ as a naively equivariant sheaf
on $X$. Note also that if $\CF$ is naively equivariant for a group
$G$, then it is also naively equivariant for any subgroup $H \subset
G$.

\begin{lemma} Suppose that $\CF \in D^b_G(X,k)$. Then $\For(\CF)$ is
  naively equivariant for $G$.
\end{lemma}

\begin{proof} Consider the quotient map $q : X \times EG \to X
  \times_G EG$. Then $q^* \CF$ is naively equivariant for $G$. Then
  smooth base change applied to the projection $p : X \times EG \to X$
  yields that $\For(\CF)=p_* q^* \CF$ is naively equivariant for $G$.
  \end{proof}

We can now prove the attractive proposition:

\begin{proof}[Proof of Proposition \ref{prop:attr}] The above arguments reduce the proof of the above to
  showing that, if $\CF \in D^b(X,k)$ is naively equivariant for the
  action of a one dimensional torus  which contracts $X$ onto $x
  \in X$, then $\pi_* \CF \to i^* \CF$ is an isomorphism. But this is
  shown in \cite{SpIH} (see also \cite{Br} for another
  account of this argument).
\end{proof}

\section{The localisation homomorphism}

Throughout this section we assume that $k$ is a unique factorisation domain and
 that $X$ is a normal complex algebraic
variety (endowed with its metric topology), acted upon algebraically by
a complex torus $T\cong (\DC^\times)^r$. In addition, we assume the following:

\begin{itemize}
\item[(A1)] \label{A1} The torus acts on $X$ with only finitely many zero- and one-dimensional
orbits and the closure of each one-dimensional orbit is smooth.
\item[(A2)] \label{A3} $X$ admits a covering by open affine connected $T$-stable subvarieties, each of which contains an attractive (hence unique) fixed point.
\end{itemize}
Note that, by a result of Sumihiro (see \cite{Sum, KKLV}),  $X$ has a covering by open affine $T$-stable subvarieties, hence (A2) is
automatically satisfied if $X$ is proper and each $T$-fixed point is
attractive. 

Let $X^T\subset X$ be the subspace of $T$-fixed points and $\CF\in D^b_T(X,k)$. The restriction homomorphism associated to the inclusion $X^T\inj X$,
$$
\THyp^\bullet(\CF)\to \THyp^\bullet(\CF_{X^T}),
$$ is called the
{\em localisation homomorphism}.

As $X^T$ is
a finite set we have $\THyp^\bullet(\CF_{X^T})=
\bigoplus_{x\in X^T}\THyp^\bullet(\CF_{x})$. Following the results of \cite{ChangSkjel}
and \cite{GKM} we will show that for certain choices of $X$, $k$ and $\CF$ the localisation map is
injective and give an explicit description of its image. This is conveniently phrased in terms of {\em moment graphs} (cf.~\cite{BrM}), as  it turns out that this
image is determined by the restriction of $\CF$ to the
one-dimensional $T$-orbits in $X$.

\subsection{One-dimensional orbits}
Suppose that $E\subset X$ is a one-di\-men\-sio\-nal $T$-orbit. Then $E\cong
T/\Stab_T(x)$ for any $x\in E$. Now $\Stab_T(x)$ is the kernel of a
character $\alpha_E\in X^\ast(T)$ which is well-defined up to
a sign. From now on we fix a choice of $\alpha_E$ for each
one-dimensional orbit $E$ in $X$. Nothing that follows depends on
this choice. 

As before we denote by $S_k$ the $\DZ$-graded symmetric
algebra of the free $k$-module
$X^\ast(T)\otimes_\DZ k$  and identify it with the $T$-equivariant cohomology
of a point with coefficients in $k$. Given $\alpha\in X^\ast(T)$ we often abuse notation and denote by $\alpha$ as well the image of
$\alpha\otimes 1\in X^\ast(T)\otimes_\DZ k$ in $S_k$.

Now $\alpha_E$ acts as zero on $H^\bullet_T(E,k)$ (see, for example, \cite[Section 1.9]{Jan}). As $\THypgr(\CF_E)$ is a $H_T^\bullet(E,k)$-module, we conclude:

\begin{lemma}\label{lemma-alphaann} For any one-dimensional $T$-orbit $E$ in $X$ and any $\CF\in D_T^b(X,k)$ we have $\alpha_E\THypgr(\CF_E)=0$.
\end{lemma}

\subsection{The localisation theorem -- part I}

 For any closed connected subgroup $\Gamma$ of $T$ we let $X^\Gamma$
 be the subset of $\Gamma$-fixed points in $X$. Let us fix a closed
 subspace $Z \subset X$ which is a discrete union
\[
Z = X^{\Gamma_1} \sqcup \dots \sqcup X^{\Gamma_n} 
\]
 of the fixed points in $X$ of finitely many connected subtori $\Gamma_1,
 \dots, \Gamma_n \subset T$. We set 
$$
P^Z:=\left\{\alpha_E\in X^\ast(T)\left|\,
\begin{matrix}
 \text{$E$ is a one-dimensional}\\
\text{$T$-orbit in $X\setminus Z$}
\end{matrix}
\right.
\right\}
$$
and define
$$
s^{Z}:=\prod_{\alpha\in P^Z}\alpha\in S_k.
$$
In addition to (A1) and (A2) we assume from now on:
\begin{itemize}
\item[(A3)] for each one-dimensional
orbit $E$ in $X$ the image of $\alpha_E\in X^\ast(T)$ is non-zero in
$S_k$.
\end{itemize}
(Of course this condition is vacuous if the characteristic of $k$ is $0$.)

We now come to the first part of
the localisation theorem. In the characteristic 0 case it is due to
Goresky, Kottwitz and MacPherson (cf.~\cite{GKM}). 

 \begin{theorem}\label{theorem-loc1} Assume that the assumptions (A1), (A2) and (A3) hold and let $\CF\in D_T^b(X,k)$. Suppose that
  $\THyp^\bullet(\CF)$ is a graded free $S_k$-module. 
Then the restriction homomorphism
$$
\THyp^\ast(\CF)\to \THyp^\ast(\CF_{Z})
$$
is injective and becomes an isomorphism after inverting $s^{Z}\in
S_k$, i.e.~after applying the functor $\cdot\otimes_{S_k} S_k[1/s^{Z}]$.
\end{theorem}

The proof of the theorem will take up the rest of this section. We follow Brion's account \cite[Section 2]{B} of the characteristic $0$ case quite closely, but at points some additional care is needed.

Let $K\cong (S^1)^r
\subset T\cong (\DC^\times)^r$ be the maximal compact subtorus of $T$. We can regard $X$ as
a $K$-space via restriction of the action. This yields a functor
$$
\res_K^T : D_T^b(X,k) \to D_K^b(X,k).
$$
As $T/K$ is contractible, for any equivariant sheaf
$\CG \in D^b_T(X,k)$ restriction gives an isomorphism
\[ \mathbb{H}_T^\bullet(\mathcal{G}) \stackrel{\sim}{\to}
\mathbb{H}_{K}^\bullet( \res_{K}^{T} \mathcal{G}). \]
In particular, we have a canonical isomorphism $H^\bullet_K(pt,k) \cong S_k$.
In the following we write $\mathbb{H}_{K}^\bullet( \mathcal{G})$ for
$\mathbb{H}_{K}^\bullet( \res_{K}^{T} \mathcal{G})$. Hence, for the proof of Theorem \ref{theorem-loc1},  it is enough to consider the restriction homomorphism
$$
\mathbb{H}_{K}^\bullet(\CF)\to \mathbb{H}_{K}^\bullet(\CF_{Z})
$$ 
and to show that it is 
injective and becomes an isomorphism after inverting $s^Z$.

Before we prove this we need a couple of preliminary results. We state them for the $K$-equivariant cohomology, however all lemmata except
Lemma \ref{lem:limit} are true with $T$ in place of $K$.

First we assume that $X=V$ is a finite dimensional $T$-module. Let
$P \subset X^*(T)$ be the characters occurring in $V$ and
$s = \prod_{\chi \in P} \chi \in S_k$ their product. Here is the first
step towards the localization theorem.

\begin{lemma} \label{lem:vec}
If $\mathcal{F} \in D_K^b(V \setminus \{ 0 \},k )$
 then $\mathbb{H}^{\bullet}_K(\mathcal{F})$ is annihilated by a power of $s$.
\end{lemma}

\begin{proof} Fix an isomorphism
\begin{equation} \label{eq:fixediso}
V \cong
\mathbb{C}_{\chi_1} \oplus \mathbb{C}_{\chi_2}
\oplus \dots \oplus \mathbb{C}_{\chi_m} \end{equation}
where $\chi_1, \chi_2, \dots, \chi_m \in P$. (Here, given $\chi \in
X^*(T)$, $\mathbb{C}_{\chi}$ denotes the one-dimensional $T$-module with
character $\chi$.)
We will use this isomorphism to write elements of $V$
as $(x_j)_{1 \le j \le m}$.
For any $1 \le i \le m$ consider the subset
\[ U_i = \{ (x_j) \in V \; | \; x_i \ne 0 \}. \]
Projection gives us an equivariant map
$U_i \to \mathbb{C}_{\chi_i}^{\times}$.
By Lemma \ref{lemma-alphaann}, the equivariant cohomology  $\mathbb{H}^{\bullet}_K(\mathcal{G})$ of each
$\mathcal{G} \in D^b_K(U_i,k)$ is annihilated by $\chi_i$.

However, $V \setminus \{ 0 \}$ is covered by the sets $U_i$
for $1 \le i \le m$ and the Mayer--Vietoris sequence allows us
to conclude that $\mathbb{H}^{\bullet}_K(\mathcal{F})$ is
annihilated by a power of  $s$. 
\end{proof}

Now let $Z\subset X$ be as before. From the above we deduce the second step:

\begin{lemma} \label{lem:annglobal}
If $\mathcal{F} \in D_K^b(X \setminus {Z},k)$ then
$\mathbb{H}_K^\bullet(\mathcal{F})$ is annihilated by a power of $s^Z$.
\end{lemma}

\begin{proof} First we assume that $X$ is affine and connected and
  contains an attractive fixed point. In this case $Z$ is necessarily
  of the form $X^{\Gamma}$ for a closed subtorus $\Gamma \subset T$. 
  We recall an argument due to
  Brion (cf. \cite[Proposition 3.2.1-1]{Ar3}, or the proof of Theorem
  17 in \cite{B}) which constructs a finite $T$-equivariant map
  \[ \pi: X \to V, \] where $V$ is a $T$-module with weights
  corresponding bijectively to the one-dimensional orbits of $T$ in
  $X$. Brion's construction is as follows:

  For each one-dimensional orbit $E \subset X$, $\overline{E}$ is
  smooth and hence isomorphic, as a $T$-space, to
  $\mathbb{C}_{\alpha_E}$. For each such orbit we may find a regular
  function $\pi_E : X \to \mathbb{C}_{\alpha_E}$ such that the
  restriction of $\pi_E$ to $\overline{E}$ is an equivariant 
  isomorphism of affine
  spaces.  Taking the direct sum over all such $\pi_E$ yields a map
  \[ X \stackrel{\pi}{\to} V := \bigoplus_E \mathbb{C}_{\alpha_E}. \]
  We claim that $\pi$ is finite. Because $x \in X$ is attractive, we
  can find a rank one subtorus of $T$ inducing a positive grading on
  the regular functions on $X$. By the graded Nakayama lemma $\pi$ is
  finite if and only if $\pi^{-1}(0)$ is finite.  If $\pi^{-1}(0)$ is
  not finite, then it contains a one-dimensional $T$-orbit (again by
  the attractiveness of $x$), but this contradicts the construction.

  Now let $V^\Gamma\subset V$ be the subspace of $\Gamma$-fixed
  points. Because each fibre of $\pi$ is finite and $\pi$ is
  equivariant it follows that
  $\pi^{-1}(V^{\Gamma}) = X^{\Gamma}$. Choose a decomposition
  \[ V = V^{\prime} \oplus V^{\Gamma} \] of $T$-modules and let $V \to
  V^{\prime}$ denote the projection.  We get an induced map
  \[ \pi^{\prime}: X \setminus X^{\Gamma} \to V^{\prime} \setminus \{
  0 \} \] and the result follows from Lemma \ref{lem:vec} because
  \[ \mathbb{H}_K^\bullet(\mathcal{F}) \cong
  \mathbb{H}_K^\bullet(\pi^{\prime}_* \mathcal{F}). \] Hence we proved
  the lemma in the case of affine $X$.

  By our assumption (A2), the general case follows from the
  Mayer--Vietoris sequence.
\end{proof}

\begin{lemma} \label{lem:limit} For any equivariant sheaf
$\mathcal{F} \in D^b_K(X,k)$ we have an isomorphism
\[ \mathbb{H}_K^\bullet(\mathcal{F}_{Z}) \cong \lim_{\to}
\mathbb{H}_K^\bullet(\mathcal{F}_{U}), \] where the direct limit takes
place over all $K$-stable open neighbourhoods $U$ of $Z$.
\end{lemma}

\begin{proof}
  By assumption $X$ has a covering by open subvarieties, all
  isomorphic to closed subvarieties of affine spaces with linear
  $T$-actions.  Thus we may choose a basis of open neighbourhoods $\{
  U_i\}_{i \in I}$ of $Z$ which are $K$-stable. (This is
  where we need the compactness of $K$.)

  Now we may write $EK$ as a countable direct limit of (finite
  dimensional) manifolds with free $K$-action (for example, by taking
  $EK = ET$ as in Section \ref{sec:torus}). Hence $X_K$ can be written
  as a countable union of compact subsets. Because $X_K$ is regular,
  we conclude that $X_K$ is paracompact (cf. \cite[Section
  5.8]{MilnorStasheff} and \cite[Theorem 6.5]{Dug}).  It is
  straightforward to see that $\{(U_i)_K\}_{i\in I}$ give a basis of open
  neighbourhoods of $Z_K$. It then follows from \cite[Remark
  2.6.9]{KS1} that we have an isomorphism
\[
\mathbb{H}_K^{\bullet}(\mathcal{F}_{Z}) =
\mathbb{H}^{\bullet}(\mathcal{F}_{Z_K})
\cong \lim_{\to} \mathbb{H}^{\bullet}(\mathcal{F}_{(U_i)_K})
= \lim_{\to} \mathbb{H}_K^\bullet(\mathcal{F}_{U})
\]
as claimed. \end{proof}



Now we are ready to prove Theorem \ref{theorem-loc1}.

\begin{proof} Let $\mathcal{F} \in D^b_T(X,k)$ and assume that
$\mathbb{H}_T^\bullet(\mathcal{F})$ is free as an $S_k$-module. We have to
show that the restriction map
\[ \mathbb{H}^{\bullet}_K(\mathcal{F}) \to
\mathbb{H}^{\bullet}_K(\mathcal{F}_{Z})
\]
is injective, and becomes an isomorphism after
inverting $s^{Z}$.

Let $U$ be an open $K$-stable neighbourhood
of $Z \subset X$.
We have inclusions
\[ U \stackrel{j}{\hookrightarrow} X \stackrel{i}{\hookleftarrow}
X \setminus U \]
and hence a distinguished triangle:
\[ i_!i^! \mathcal{F} \to \mathcal{F} \to j_* j^* \mathcal{F}
\stackrel{[1]}{\to}. \] Applying Lemma \ref{lem:annglobal} (and
remembering that $i_* \cong i_!$) we deduce that
$\mathbb{H}_{K}^\bullet(i_!i^! \mathcal{F})$ is annihilated by a power
of $s^Z$. As $\mathbb{H}^{\bullet}_{K}(\mathcal{F})$ is free,
the restriction map $\mathbb{H}^{\bullet}_{K}(\mathcal{F}) \to
\mathbb{H}^{\bullet}_{K}(\mathcal{F}_U)$ is injective. It also follows
that it becomes an isomorphism after inverting $s^Z$.

To finish the proof, note that, by Lemma \ref{lem:limit},
\[
\mathbb{H}^{\bullet}_{K}(\mathcal{F}_{Z}) \cong \lim_{\to}
\mathbb{H}^{\bullet}_{K}(\mathcal{F}_U).
\]
Because $\mathbb{H}^{\bullet}_{K}(\mathcal{F}) \to
\mathbb{H}^{\bullet}_{K}(\mathcal{F}_U)$ is injective for all $U$ it follows that
$\mathbb{H}^{\bullet}_{K}(\mathcal{F}) \to
\mathbb{H}^{\bullet}_{K}(\mathcal{F}_{Z})$ is injective. Lastly, this map becomes
an isomorphism after inverting $s^{Z}$ because the direct
limit commutes with tensor products.
\end{proof}

\section{The image of the localisation homomorphism}

We are now going to describe the image of the localisation homomorphism under a
certain further restriction on the ring $k$ which is called the {\em GKM-condition}. For this it is convenient to use
the language of sheaves on moment graphs. We start by recalling
the main definitions and constructions in the theory of moment
graphs. In particular, we define the $\DZ$-graded category $\CG\catmod^\DZ_k$ of
$k$-sheaves on a moment graph $\CG$ and associate to any such sheaf
$\SF$ its space of global sections $\Gamma(\SF)$.

To a $T$-space $X$ with finitely many zero- and one-dimensional
orbits we associate a moment graph $\CG_X$ and define a functor
$$
\DW\colon D_T^b(X,k)\to\CG_X\catmod^\DZ_k
$$
between $\DZ$-graded categories.
We then show that
under some assumptions on $\CF\in D_T^b(X,k)$,  the equivariant cohomology of $X$ with coefficients in $\CF$
 coincides with the space of global sections  of  $\DW(\CF)$,
 i.e.
$$
\THypgr(\CF)=\Gamma(\DW(\CF)).
$$

\subsection{Sheaves on moment graphs}
 Let $Y\cong \DZ^r$ be a lattice of finite rank.

\begin{definition} An (unordered) {\em moment graph} $\CG$ over $Y$ is given by the following data:
\begin{itemize}
\item A graph $(\CV,\CE)$ with set of vertices $\CV$ and set of edges $\CE$.
\item A map $\alpha\colon \CE\to Y\setminus\{0\}$.
\end{itemize}
\end{definition}
We assume that two vertices of a moment graph are connected by at most one edge.

Let $\CG=(\CV,\CE,\alpha)$ be a moment graph. We write $E\colon x\linie y$ for an edge $E$ that connects the vertices $x$ and $y$. If we also want to denote the label $\alpha=\alpha(E)$ of $E$, then we write $E\colon x\stackrel{\alpha}\llinie y$. As before we denote by $S_k=S(Y\otimes_\DZ k)$ the symmetric algebra of $Y$ over $k$, which  we consider as a graded algebra with $Y\otimes_\DZ k$ sitting in degree $2$. 

\begin{definition}
A {\em $k$-sheaf} $\SM$ on a moment graph $\CG$ is given by the following data:
\begin{itemize}
\item a graded $S_k$-module $\SM^x$ for any vertex $x\in\CV$,
\item a graded $S_k$-module $\SM^E$ with $\alpha(E)\SM^E=0$ for any $E\in\CE$,
\item a homomorphism $\rho_{x,E}\colon \SM^x\to \SM^E$ of graded $S_k$-modules for any vertex $x$ lying on the edge $E$.
\end{itemize}
\end{definition}

For a $k$-sheaf $\SM$ on $\CG$ and $l\in\DZ$ we denote by $\SM[l]$ the shifted $k$-sheaf with stalks $\SM[l]^x=(\SM^x)[l]$, $\SM[l]^E=(\SM^E)[l]$ and shifted $\rho$-homomorphisms. A morphism $f\colon \SM \to \SN$ between $k$-sheaves $\SM$ and $\SN$ on $\CG$ is given by a collection of homomorphisms of graded $S_k$-modules $f^x\colon \SM^x\to \SN^x$ and $f^E\colon \SM^E\to \SN^E$ for all vertices $x$ and edges $E$ that are compatible with the maps $\rho$, i.e. such that the diagram

\centerline{
\xymatrix{
\SM^x\ar[d]_{\rho^\SM_{x,E}}\ar[r]^{f^x}&\SN^x\ar[d]^{\rho^\SN_{x,E}}\\
\SM^E\ar[r]^{f^E}&\SN^E}
}
\noindent
commutes for all vertices $x$ that lie on the edge $E$. We denote by $\CG\catmod_k^\DZ$ the category whose objects are $k$-sheaves on $\CG$ and whose morphisms are the morphisms between $k$-sheaves. It is $\DZ$-graded by the functor $\SM\mapsto\SM[1]$.

\subsection{Sections of sheaves and the structure algebra}
The {\em structure algebra} over $k$ of a moment graph $\CG$ is
$$
\CZ_k=\left\{(z_x)\in\prod_{x\in\CV}S_k\left|
\begin{matrix}
\text{ $z_x\equiv z_y\mod \alpha(E)$ }\\
\text{ for all edges $E\colon x\linie y$ }
\end{matrix}
\right\}.
\right.
$$
Coordinatewise addition and multiplication makes $\CZ_k$ into an
$S_k$-algebra. It is $\DZ$-graded if we consider the product in the definition in the graded sense.

Let $\SM$ by a $k$-sheaf on $\CG$. For any subset $\CI$ of $\CV$ we
define the {\em space of sections} of $\SM$ over $\CI$ by
$$
\Gamma(\CI,\SM):=\left\{(m_x)\in\prod_{x\in\CI} \SM^x\left|\begin{matrix}
      \rho_{x,E}(m_x)=\rho_{y,E}(m_y) \\
\text{ for all edges $E\colon x\linie y$}\\
\text{ with $x,y\in\CI$}
\end{matrix}\right.\right\}.
$$
Coordinatewise multiplication makes $\Gamma(\CI,\SM)$ into a
$\CZ_k$-module (as $\alpha(E)\rho_{x,E}(m_x)=0$ for any edge $E$ with
vertex $x$). Again it is a graded module when the product is taken in the category of graded $S_k$-modules.

We call the space $\Gamma(\SM):=\Gamma(\CV,\SM)$ the space of {\em global sections}. If $\CI\subset \CJ$ are two subsets of $\CV$, then the canonical projection $\bigoplus_{x\in \CJ}\SM^x\to \bigoplus_{x\in \CI}\SM^x$ induces a restriction map $\Gamma(\CJ,\SM)\to\Gamma(\CI,\SM)$.
 
\subsection{The costalks of a sheaf}
Let $\SM$ be a $k$-sheaf on $\CG$ and let $x$ be a vertex. Then we define the {\em costalk} $\SM_x$ of $\SM$ at $x$ to be the $S_k$-module
$$
\SM_x:=\{ m \in \SM^x\mid \rho_{x,E}(m)=0 \text{ for all edges $E$ that contain $x$}\}.
$$
We can identify $\SM_x$ in an obvious way with the kernel of the restriction homomorphism $\Gamma(\CV,\SM)\to\Gamma(\CV\setminus\{x\},\SM)$.

\subsection{The moment graph associated to a $T$-variety}\label{subsec-mgfromvar}

To a complex $T$-variety $X$ satisfying (A1) we associate the following moment graph $\CG_X=(\CV,\CE,\alpha)$ over the lattice $X^\ast(T)$:

\begin{itemize}
\item We set $\CV:=X^T$.
\item The vertices $x$ and $y$, $x\ne y$, are connected by an edge if there is a one-dimensional orbit $E$ such that $\ol E=E\cup\{x, y\}$. We denote this edge by $E$ as well.
\item We let $\alpha(E)=\alpha_E\in X^\ast(T)$ be the chosen character.
\end{itemize}

Note that only those one-dimensional orbits $E$ in $X$ give rise to an
edge that pick up two distinct fixed points in their
closure.

 \subsection{The functor $\DW$}
Suppose that $E\subset X$ is a one-dimensional $T$-orbit, and suppose that $x\in \ol E$ is a fixed point in its closure. Let $\CF$ be an object in $D^b_T(X,k)$. Then the restriction homomorphism
$$
\THypgr(\CF_{E\cup\{x\}})\to \THypgr(\CF_x)
$$
is an isomorphism by the attractive Proposition \ref{prop:attr}.
Hence we can define a homomorphism $\rho_{x,E}$ from $\THypgr(\CF_x)$ to $\THypgr(\CF_E)$ by composing the inverse of the above homomorphism with the restriction homomorphism $\THypgr(\CF_{E\cup\{x\}}) \to \THypgr(\CF_E)$:
$$
\rho_{x,E}\colon \THypgr(\CF_x)\stackrel{\sim}\leftarrow\THypgr(\CF_{E\cup\{x\}})\to \THypgr(\CF_{E}).
$$

Now we can define the functor $\DW$.
To an equivariant sheaf $\CF\in D^b_T(X,k)$ on $X$ we associate the following $k$-sheaf $\DW(\CF)$ on $\CG_X$:
\begin{itemize}
\item For a vertex $x\in \CV$ we set $\DW(\CF)^x:=\THypgr(\CF_{x})$.
\item For a one-dimensional orbit $E$ we set
  $\DW(\CF)^E:=\THypgr(\CF_E)$ (note that $\alpha_E\THypgr(\CF_E)=0$
  by Lemma \ref{lemma-alphaann}).
\item In case that $x\in \ol E$ we let $\rho_{x,E}\colon\DW(\CF)^x\to\DW(\CF)^E$ be the map constructed above.
\end{itemize}
This construction clearly extends to a functor $\DW\colon D^b_T(X,k)\to \CG_X\catmod_k^{\mathbb{Z}}$.

\subsection{The case $X=\DP^1$}

 Suppose that $T$ acts linearly on $\DP^1$ via a non-trivial character $\alpha$.  In this case the moment graph is 
 $$
 0\stackrel{\alpha}\linie \infty.
 $$
 For $\CF\in D_T^b(\DP^1,k)$  the sheaf $\DW(\CF)$ consists of the stalks $\THypgr(\CF_0)$, $\THypgr(\CF_\infty)$ and the space $\THypgr(\CF_{\DC^\times})$ together with the maps
 $$
 \THypgr(\CF_0)\stackrel{\rho_{0,\DC^\times}}\longrightarrow \THypgr(\CF_{\DC^\times})\stackrel{\rho_{\infty,\DC^\times}}\longleftarrow \THypgr(\CF_\infty).
 $$
A consequence of the Mayer--Vietoris sequence is the following lemma.

\begin{lemma}\label{lemma-explP1}  Let $\CF\in D_T^b(\DP^1,k)$.  Then the image of the restriction homomorphism $\THypgr(\CF)\to \THypgr(\CF_0)\oplus\THypgr(\CF_\infty)$ is
$\{(z_0,z_\infty)\mid
\rho_{0,\DC^\times}(z_0)=\rho_{\infty,\DC^\times}(z_\infty)\}$.
\end{lemma}

\subsection{The localisation theorem -- part II}\label{subsec-locpartII} Now we assume that $X$ satisfies the assumptions (A1), (A2) and (A3). Let $\CF\in D_T^b(X,k)$.  If
$\THyp^\bullet(\CF)$ is a free $S_k$-module, then Theorem \ref{theorem-loc1} shows that
we can view $\THyp^\bullet(\CF)$ as a submodule of
$\bigoplus_{x\in X^T}\THyp^\bullet(\CF_x)=\bigoplus_{x\in
  X^T}\DW(\CF)^x$. The space of global sections $\Gamma(\DW(\CF))$ is
a submodule of this direct sum as well. In this section we want to prove that these two
submodules coincide.

We need some more notation.  For $\alpha\in X^\ast(T)$ let us
define $X^{\alpha}$ to be the subvariety of all $T$-fixed points
in $X$ and all one-dimensional orbits $E\subset X$ such that
$k\alpha\cap k\alpha_E\ne 0$.
%
Then $X^{0}=X^T$ for all rings $k$, but  in general $X^{\alpha}$ depends on the
ring $k$. We define
$$
P^{\alpha} :=\left\{\alpha_E\in X^\ast(T)\left|\,
\begin{matrix}
 \text{$E$ is a one-dimensional}\\
\text{$T$-orbit in $X\setminus X^{\alpha}$}
\end{matrix}
\right.
\right\}
$$
and 
$$
s^{\alpha}:=\prod_{\alpha_E\in P^{\alpha}} \alpha_E \in S_k.
$$

We need some additional assumptions on our data:
\begin{itemize}
\item[(A4a)] For any $\alpha\in X^\ast(T)$ the space $X^{\alpha}$
  is a disjoint union of points and $\DP^1$'s.
\item[(A4b)] If  $E$ is a  one-dimensional $T$-orbit and  $n\in\DZ$ is
  such that $\alpha_E$ is divisible by $n$ in $X^\ast(T)$, then $n$ is
  invertible in $k$.
\end{itemize}
Note that (A4a) and (A4b) imply that the greatest common divisor of
$s^{\alpha}$ for all $\alpha\in X^\ast(T)$ is $1$. For the proof of the
next theorem we will only need this  fact, but we need the stronger
statements (A4a) and (A4b) later.
Note also that (A4a) guarantees that we can apply Theorem
\ref{theorem-loc1} with $Z = X^{\alpha}$ and $s^Z = s^\alpha$.

Let $\CG_X$ be the moment graph associated to $X$. For $\alpha\in
X^\ast(T)$ we denote by $\CG_X^\alpha$ the moment graph obtained from
$\CG_X$ by deleting all edges $E$ with $k\alpha_E\cap k\alpha=0$. Then
(A4a) is equivalent to:
\begin{itemize}
\item[(A4a)$^\prime$] The moment graph $\CG_X^{\alpha}$ is a
  (discrete) union of moment graphs with only one or two vertices.
\end{itemize}
Now we can
state the second part of the localisation theorem.

\begin{theorem}\label{theorem-loc2} Suppose that (A1), (A2), (A3),
  (A4a) and (A4b) hold. Let $\CF\in D_T^b(X,k)$ and suppose that
  $\THyp^\bullet(\CF)$ and $\THyp^\bullet(\CF_{X^T})$ are free
  $S_k$-modules. Then
$$
\THyp^\bullet(\CF)=\Gamma(\DW(\CF))
$$
as submodules of $\bigoplus_{x\in X^T}\THyp^\bullet(\CF_x)=\bigoplus_{x\in
  X^T}\DW(\CF)^x$.
\end{theorem}

For the proof of the above statement we use similar arguments as the ones given in  \cite{ChangSkjel}, \cite{GKM} or \cite{B}. Again we follow \cite{B} closely. 

\begin{proof}
As a first step let $\CF\in D_T^b(X,k)$ be any sheaf and $\alpha\in
X^\ast(T)$. Let $\Gamma^\alpha(\DW(\CF))$ be the sections of the sheaf
$\DW(\CF)$ on the moment graph $\CG_X^\alpha$ (so we only consider the
edges $E$ with $k\alpha_E \cap k\alpha \ne 0$). By (A3),
$$
\Gamma(\DW(\CF))=\bigcap_{\alpha\in X^\ast(T)}\Gamma^\alpha(\DW(\CF)).
$$
By (A4a),  $X^\alpha$ is a  discrete union of points and $\DP^1$'s. Hence, if we
denote by $r_{\alpha}:\THypgr(\CF_{X^\alpha}) \to \THypgr(\CF_{X^T})$
the restriction map, then Lemma \ref{lemma-explP1} yields $\Gamma^\alpha(\DW(\CF))=r_{\alpha}(\THypgr(\CF_{X^\alpha}))$. Hence:
$$
\Gamma(\DW(\CF))=\bigcap_{\alpha\in X^\ast(T)} r_{\alpha}(\THypgr(\CF_{X^\alpha})).
$$
So we have to show that $\THypgr(\CF)=\bigcap_{\alpha\in X^\ast(T)} r_{\alpha}(\THypgr(\CF_{X^\alpha}))$.

Clearly $\THypgr(\CF)$ is contained in the intersection $\bigcap_{\alpha\in X^\ast(T)} r_{\alpha}(\THypgr(\CF_{X^{\alpha}}))$. Hence it remains to show that if $f \in
\mathbb{H}_T^\bullet(\mathcal{F}_{X^{T}})$ is in $r_\alpha(
\THypgr(\CF_{X^{\alpha}}))$  for all $\alpha\in
X^\ast(T)$, then $f$ is
contained in $\mathbb{H}_T^\bullet(\mathcal{F})$.

By  Theorem \ref{theorem-loc1} the injection $i : \THypgr(\CF)\to \THypgr(\mathcal{F}_{X^T})$ becomes an isomorphism after inverting $s^0$.  
By assumption, $\mathbb{H}_T^\bullet(\mathcal{F})$ is a free $S_k$-module.
We choose a basis $e_1, \dots, e_m$ for
$\mathbb{H}_T^\bullet(\mathcal{F})$ and denote by $e^*_1, \dots, e^*_m \in
\Hom (\mathbb{H}_T^\bullet(\mathcal{F}), S_k)$ the dual basis.
Because $i$ becomes an isomorphism after inverting $s^0$, we can find
$\tilde{e^*_1}, \dots, \tilde{e^*_m} \in
\Hom_{S_k}(\mathbb{H}_T^\bullet(\mathcal{F}_{X^{T}}), S_k[1/s^0])$ such that
$e^*_j = \tilde{e^*_j} \circ i$ for $1 \le j \le m$. Note that
$f$ is
in $\THypgr(\CF)$ if
and only if $\tilde{e^*_j}(f)$ is contained in $S_k$ for $1 \le j \le m$.

By Theorem \ref{theorem-loc1}, the map
\[ \mathbb{H}_T^\bullet(\mathcal{F})
\hookrightarrow \mathbb{H}_T^\bullet(\mathcal{F}_{X^{\alpha}})
\]
becomes an isomorphism after inverting $s^\alpha$. As $f$ is contained in $ \mathbb{H}_T^\bullet(\mathcal{F}_{X^{\alpha}})$, we conclude 
$\tilde{e^\ast_j}(f)\in S_k[1/s^\alpha]$
for any $1\le j\le m$. Hence, 
$$
\tilde{e^\ast_j}(f)\in \bigcap_{\alpha\in X^\ast(T)} S_k[1/s^\alpha].
$$
But $\bigcap_{\alpha\in X^\ast(T)} S_k[1/s^\alpha]=S_k$ as the
greatest common divisor of all $s^\alpha$ is $1$.
Hence $\tilde{e^\ast_j}(f)\in S_k$, 
which is what we wanted to show. 
\end{proof}

\section{Equivariant parity sheaves}\label{sec-eps}

In the following sections we consider equivariant parity sheaves on a stratified variety, which were introduced in \cite{JMW2}. It
turns out that the equivariant cohomology of such a sheaf is free over the symmetric algebra,
so by the results in the previous sections it can be calculated by
moment graph techniques. We determine the
corresponding sheaves on the moment graph explicitely: we show
that these are the sheaves that are constructed by the Braden-MacPherson algorithm (cf.~\cite{BrM}).

For all of the above, we need an additional datum: a stratification of
the variety.

\subsection{Stratified varieties}\label{subs-stratification}
We assume from now that the $T$-variety $X$ is endowed with a stratification
$$
X=\bigsqcup_{\lambda\in\Lambda} X_\lambda
$$
by $T$-stable subvarieties $X_\lambda$. We write $D_{T,
  \Lambda}^b(X,k)$ for the full subcategory of $D_T^b(X,k)$
consisting of objects which are constructible with respect to this
stratification. In addition to the assumptions (A1) and (A2) we assume:
\begin{itemize}
\item[(S1)] For each $\lambda\in\Lambda$ there is a $T$-equivariant isomorphism $X_\lambda\cong\DC^{d_\lambda}$, where $\DC^{d_\lambda}$ carries a linear $T$-action.
\item[(S2)] The category $D_{T,\Lambda}^b(X,k)$ is preserved by under
  Grothendieck-Verdier duality. 
(This is satisfied, for example, if the stratification is Whitney.)
\end{itemize}

By (A1) and (A2) each stratum $X_{\lambda}$ contains a unique fixed point. We denote this fixed point by $x_{\lambda}$.
 
The topology of $X$ gives us a partial order on the set $\Lambda$: We set $\lambda\le \mu$ if and only if $X_\lambda\subset \ol X_\mu$. We use the following notation for an arbitrary partially ordered set $\Lambda$: For $\lambda\in\Lambda$ we set $\{\ge\lambda\}:=\{\nu\in\Lambda\mid\nu\ge\lambda\}$ and we define $\{\le\lambda\}$, $\{>\lambda\}$, etc.~in an analogous fashion.

\begin{definition} Let $\CK$ be a subset of $\Lambda$.
\begin{itemize}
\item We say that $\CK$ is {\em open}, if for all $\gamma\in \CK$, $\lambda\in \Lambda$ with $\lambda\ge \gamma$ we have $\lambda\in \CK$, i.e. if $\CK=\bigcup_{\gamma\in\CK}\{\ge\gamma\}$.
\item We say that $\CK$ is {\em closed} if $\Lambda\setminus\CK$ is open, i.e. if $\CK=\bigcup_{\gamma\in\CK}\{\le\gamma\}$.
\item We say that $\CK$ is {\em locally closed} if it is the intersection of an open and a closed subset of $\Lambda$.
\end{itemize}
\end{definition}
For a subset $\CK$ of $\Lambda$ the set $\CK^+:=\bigcup_{\gamma\in\CK}\{\ge \gamma\}$ is the smallest open subset containing $\CK$, and $\CK^-:=\bigcup_{\gamma\in\CK}\{\le\gamma\}$ is the smallest closed subset containing $\CK$. $\CK$ is locally closed if $\CK=\CK^-\cap\CK^+$.

For any subset $\CK$ of $\Lambda$ we define
$$
X_\CK=\bigsqcup_{\gamma\in \CK} X_\gamma \subset X
$$
If $\CK$ is open (closed, locally closed), then $X_\CK$
is an open (closed, locally closed, resp.) subvariety in $X$. In
particular, for any $\lambda\in\Lambda$ the subvariety $X_{\le
  \lambda}:=X_{\{\le \lambda\}}$ is closed. For
$\CF\in D_{T,\Lambda}^b(X,k)$ we define $\CF_\CK:=\CF_{X_\CK}$.

\subsection{Equivariant parity sheaves}
For $\lambda\in\Lambda$ we denote by $i_\lambda\colon X_\lambda\to X$ the inclusion. We now give the definition of an equivariant parity sheaf on $X$:

\begin{definition} Let $?$ either denote the symbol $\ast$ or the symbol $!$, and let $\CP\in D_T^b(X,k)$.
\begin{itemize}
\item $\CP$ is {\em $?$-even} (resp. {\em $?$-odd}) if for all $\lambda\in \Lambda$ the sheaf $i^?_\lambda\CP$ is a direct sum of constant sheaves appearing only in even (resp. odd) degrees.
\item $\CP$ is {\em even} (resp. {\em odd}) if it is both $*$-even and $!$-even (both $*$-odd and $!$-odd, resp.).
\item $\CP$ is {\em parity} if it may be written as a sum $\CP = \CP_{0} \oplus \CP_{1}$ with $\CP_{0}$ even and $\CP_{1}$ odd.
\end{itemize}
\end{definition}

Note that, by assumption (S1), for all $\lambda \in \Lambda$, all $T$-equivariant local systems on $X_{\lambda}$ are trivial and we have
\[
H^{\bullet}_T(X_{\lambda},k) = H^{\bullet}_T(pt,k) = S_k.
\]
Hence, we have the following classification of indecomposable parity sheaves (see \cite[Theorem 2.9]{JMW2}):

\begin{theorem} \label{thm-parityunique} Suppose that $k$ is a complete local ring. For all $\lambda \in \Lambda$ there exists, up to isomorphism, at most one
indecomposable parity sheaf $\CP(\lambda)$ extending the equivariant constant sheaf $\ul{k}_{X_{\lambda}}$. 
Moreover, any indecomposable parity sheaf is isomorphic to 
$\CP(\lambda)[i]$ for some $\lambda \in \Lambda$ and some integer $i$.
\end{theorem}

Note that in this paper (in contrast to \cite{JMW2}) we normalise
indecomposable parity sheaves in such a way that the restriction of
$\CP(\lambda)$ to $X_{\lambda}$ is the constant sheaf in degree
$0$. Also, in \cite{JMW2} parity sheaves a considered with respect to
an arbitrary ``pariversity'' $\dagger : \Lambda \to
\mathbb{Z}/2\mathbb{Z}$.  In this paper we only consider parity
sheaves with respect to the constant pariversity, which corresponds to
the above definition.

\begin{proposition} \label{prop:parityselfdual}
Let $\lambda \in \Lambda$ and assume that $\CP(\lambda)$ exists. We have $\DD( \CP(\lambda)) \cong \CP(\lambda)[2d_{\lambda}] $ where $d_{\lambda}$ denotes the complex dimension of $X_{\lambda}$.
\end{proposition}

\begin{proof} This is a simple consequence of the above theorem, together with the fact that $\DD$ preserves parity and the fact that $\DD \ul{k}_{X_{\lambda}} \cong \ul{k}_{X_{\lambda}}[2d_{\lambda}]$.
\end{proof}

\subsection{Short exact sequences involving parity sheaves}

Let $\CQ$ be a parity sheaf on $X$ and let $\CJ\subset\Lambda$ be an open subset with closed complement $\CI=\Lambda\setminus \CJ$. Denote by $j\colon X_\CJ\to X$ and $i\colon X_\CI\to X$ the corresponding inclusions. Consider the distinguished triangle
\begin{equation} \label{eq:!-exact}
i_!i^!\CQ\to \CQ\to j_\ast j^\ast\CQ\stackrel{[1]}\to.
\end{equation}
\begin{lemma}\label{lemma-ses}
\begin{enumerate}
\item The above triangle gives rise to a short exact sequence
$$
0\to\THypgr(i^!\CQ)\to\THypgr(\CQ)\to\THypgr(\CQ_{\CJ})\to 0.
$$
\item Let $\CP$ be another parity sheaf on $X$. Then the above triangle gives rise to a short exact sequence
$$
0\to \Hom^\bullet(i^\ast\CP,i^!\CQ)\to\Hom^\bullet(\CP, \CQ)\to\Hom^\bullet(\CP_{\CJ},
\CQ_{\CJ})\to 0.
$$
\end{enumerate}
\end{lemma}
\begin{proof} We may assume without loss of generality that $\CQ$ is
  even. Then the distinguished triangle in \eqref{eq:!-exact} is a
  distinguished triangle of $!$-even sheaves. If $\CP$ (resp. $\CQ'$)
  is $*$-even (resp. $!$-even) then an induction (see \cite[Corollary
  2.8]{JMW2}) shows that $\Hom(\CP, \CQ'[n]) = 0$ for odd $n$. Then (2) follows and part (1)  is the case $\CP = \ul{k}_X$.
\end{proof}

\subsection{Further properties of equivariant parity sheaves}

The following properties of the equivariant cohomology of parity sheaves will be useful when we come to relate parity sheaves and Braden-MacPherson sheaves in the next section.

\begin{proposition}\label{prop-propofparity} Suppose that $\CP$ is an equivariant parity sheaf on $X$. Then the following holds:
\begin{enumerate}
\item For any open subset $\CJ$ of $\Lambda$ the equivariant cohomology $\THypgr(\CP_\CJ)$ is a free $S_k$-module.
\item For any open subset $\CJ$ of $\Lambda$ the restriction homomorphism $\THypgr(\CP)\to \THypgr(\CP_\CJ)$ is surjective.
\item Assume that (A4b) holds and suppose that $E\subset X_\lambda$ is a one-dimensional $T$-orbit. Then the restriction map
$$
\rho_{\lambda,E}\colon \THypgr(\CP_{x_\lambda})\to \THypgr(\CP_{E})
$$
is surjective with kernel $\alpha_E \THypgr(\CP_{x_\lambda})$.
\end{enumerate}
\end{proposition}
\begin{proof} 
Note that (2) has already been shown in the previous lemma. For (1), first note that  if we choose an open subset $\CJ \subset \Lambda$ then $\CP_{\CJ}$ is a parity sheaf on $X_{\CJ}$. Hence it is enough to show that $\THypgr(\CP)$ is a free $S_k$-module. Choose $x \in \Lambda$ minimal, let $\CI=\{x\}$ and $\CJ=\Lambda\setminus\{x\}$. We have an exact sequence
\[
0\to\THypgr(i^!\CP)\to\THypgr(\CP)\to\THypgr(\CP_{\CJ})\to 0.
\]
As $\CP$ is a parity sheaf, $i^! \CP$ is a direct sum of constant sheaves and so $\THypgr(i^!\CP)$ is a free $S_k$-module. Using induction we can assume that  $\THypgr(\CP_{\CJ})$ is a free $S_k$-module. Hence $\THypgr(\CP)$ is free.

Let us prove (3). Since $E\cup\{x_\lambda\}$ is contained in $X_\lambda$, the restriction of $\CP$ to $E\cup\{x_\lambda\}$ is isomorphic to a sum of shifted constant sheaves. Hence it is enough to show that if $T$ acts on $\DC$ via the character $\alpha\ne 0$ such that $n$ is invertible in $k$ if $\alpha$ is divisible by $n$ in $X^\ast(T)$, then the map
 $$
\rho_{0,\DC^\times}\colon H_T^\bullet(\{0\},k)\to H_T^\bullet(\DC^\times,k)
$$
identifies with the canonical quotient map $S_k\to S_k/\alpha S_k$.
 With characteristic $0$ coefficients this is proved in
\cite[Section 1.10]{JanNotes}. The divisibility assumption guarantees
that the argument given there also works with coefficients in $k$.
\end{proof}

%



\subsection{Obtaining parity sheaves via resolutions} Up until now we have only discussed various properties of parity sheaves, without discussing their existence. We now show that the existence of certain proper morphisms to the varieties $\overline{X_{\lambda}}$ guarantees the existence of parity sheaves.

Assume that, for all $\lambda \in \Lambda$, there exists a $T$-variety $\widetilde{X_{\lambda}}$ and a proper surjective morphism
\[ \pi_{\lambda} : \widetilde{X_{\lambda}} \to \overline{X_{\lambda}} \]
such that:
\begin{itemize}
\item[(R1)] each $\widetilde{X_{\lambda}}$ is smooth and admits a
  $T$-equivariant closed
  embedding $\widetilde{X_{\lambda}} \hookrightarrow \mathbb{P}(V)$ for some
  $T$-module $V$,
\item[(R2)] the fixed point set ${\widetilde{X_{\lambda}}}^T$ is finite,
\item[(R3)] $\pi_{\lambda*} \underline{k}_{\widetilde{X_{\lambda}}}$ is constructible with respect to the stratification $\Lambda$ (that is, $\pi_{\lambda*} \underline{k}_{\widetilde{X_{\lambda}}} \in D^b_{T, \Lambda}(X, k)$).
\end{itemize}
Note that we do not assume that the morphisms $\pi_{\lambda}$ are birational.

\begin{theorem} \label{thm:parityexists} Assume that $k$ is a complete local principal ideal domain. With the above assumptions we have:
\begin{enumerate}
\item For all $\lambda \in \Lambda$ there exists an indecomposable parity sheaf $\CP(\lambda) \in D^b_T(X,k)$ with support equal to $\overline{X_{\lambda}}$ and such that $\CP(\lambda)_{X_{\lambda}} \cong \ul{k}_{X_{\lambda}}$.
\item For all $\mu \le \lambda$ the restriction homomorphism
\[
\THypgr( \CP(\lambda)) \to \THypgr( \CP(\lambda)_{x_{\mu}}) 
\]
is surjective.
\item The cohomology $\THypgr( \CP(\lambda) )$ is self-dual of degree
  $2 \dim X_{\lambda}$. That is,
\[
\Hom_{S_k}^\bullet (\THypgr( \CP(\lambda)), S_k) \cong \THypgr(
\CP(\lambda))[2\dim X_{\lambda}].
\]
\end{enumerate}
\end{theorem}

Before proving the theorem we state and prove two propositions. For this we need some more notation. Given a $T$-variety $Z$ we write $\omega_Z$ for the $T$-equivariant dualising sheaf in $D^b_T(Z,k)$. Note that, up to reindexing, $\THypgr(\omega_{Z})$ is the $T$-equivariant Borel-Moore homology of $Z$.

Let us fix $\mu \le \lambda$ and set $F := \pi_\lambda^{-1}(x_{\mu})$. We have:

\begin{proposition} \label{prop:BB}
\begin{enumerate}
\item $\THypgr(\omega_F)$ and $\THypgr(\omega_{\widetilde{X_{\lambda}}})$
are free $S_k$-modules concentrated in even degrees.
\item The canonical map $\THypgr(\omega_F) \to \THypgr(\omega_{\widetilde{X_{\lambda}}})$ is a split injection of $S_k$-modules.
\end{enumerate}
\end{proposition}

\begin{proof}
As $x_{\mu}$ is attractive, there exists a one parameter subgroup $\gamma\colon  \DC^\times \to T$ which contracts an open neighbourhood of $x_{\mu}$ in $X$ onto $x_{\mu}$ as $z \in \DC^\times$ goes to $0$. Moreover, we can choose $\gamma$ such that $\widetilde{X_{\lambda}}^{\DC^\times} = \widetilde{X_{\lambda}}^{T}$.

Now consider the Bialynicki-Birula's minus decomposition of $\widetilde{X_{\lambda}}^{T}$ with respect to $\gamma$. That is, for each $x \in \widetilde{X_{\lambda}}^T$ set
\[
Y_x^- := \{ y \in \widetilde{X_{\lambda}} \; | \; \lim_{z \to \infty} \gamma(z) \cdot y = x \}.
\]
Then a theorem of Bialynicki-Birula (\cite{BB}) asserts that each $Y_x^-$ is a locally closed $T$-stable subvariety of $\widetilde{X_{\lambda}}$ isomorphic to an affine space. Our choice of $\gamma$ implies that
\[
F = \pi_{\lambda}^{-1}(x_{\mu}) = \bigsqcup_{x \in F^T} Y_x.
\]

Moreover, by assumption we can find a $T$-equivariant embedding
\[
\widetilde{X_{\lambda}} \hookrightarrow \mathbb{P}(V)
\]
and we may decompose $V \cong \bigoplus V_i$ where $V_i$ denotes the
$i^{th}$ weight space of the $\mathbb{C}^*$-action on $V$ given by
$\gamma$. If we set $V_{\le i} = \bigoplus_{j \le i} V_i$ then it is
straightforward to check that the filtration
\[
\emptyset \subset \dots \subset \mathbb{P}(V_{\le i}) \subset \mathbb{P}(V_{\le i+1}) 
\subset  \dots \subset \mathbb{P}(V)
\]
induces filtrations of $\widetilde{X_{\lambda}}$ and $F$ by closed
subvarieties such that each successive difference is a disjoint union
of Bialynicki-Birula cells. A simple induction (see, for example,
\cite{Fulton} for the non-equivariant case) shows that both
$\THypgr(\omega_{\widetilde{X_{\lambda}}})$ and $\THypgr(\omega_F)$
are free $S_k$-modules with basis given by the equivariant fundamental
classes of closures of the Bialynicki-Birula cells. The two
statements of the lemma then follow.
\end{proof}

\begin{proposition} \label{prop:BBres}
With notation as above we have:
\begin{enumerate}
\item $\pi_{\lambda *} \ul{k}_{\widetilde{X_{\lambda}}}$ is parity and its support is equal to $\overline{X_{\lambda}}$.
\item For all $\mu \le \lambda$ the restriction homomorphism
\[
\THypgr(\pi_{\lambda *} \ul{k}_{\widetilde{X_{\lambda}}})
\to \THypgr( (\pi_{\lambda *} \ul{k}_{\widetilde{X_{\lambda}}})_{x_{\mu}}) 
\]
is surjective.
\end{enumerate}

\end{proposition}

\begin{proof} The support claim follows from the surjectivity of $\pi_{\lambda}$. We now explain why $\pi_{\lambda*} \ul{k}_{\widetilde{X_{\lambda}}}$ is parity. As $\pi_{\lambda}$ is proper and $\widetilde{X_\lambda}$ is smooth, $\pi_{\lambda*} \ul{k}_{\widetilde{X_{\lambda}}}$ is self-dual up to a shift and so it is enough to show that $\pi_{\lambda*} \ul{k}_{\widetilde{X_{\lambda}}}$ is $!$-even. As $\pi_{\lambda*} \ul{k}_{\widetilde{X_{\lambda}}}$ is constructible with respect to the $\Lambda$-stratification, it is enough to show that, for all $\mu$, $i_{x_{\mu}}^! \pi_{\lambda*} \ul{k}_{\widetilde{X_{\lambda}}}$ is a direct sum of constant sheaves concentrated in even degrees, where $i_{x_\mu}$ denotes the inclusion $ \{ x_{\mu} \} \hookrightarrow X$. A devissage argument shows that this is the case if and only if $\THypgr( i_{x_{\mu}}^! \pi_{\lambda*} \ul{k}_{\widetilde{X_{\lambda}}})$ is a free $S_k$-module.

By proper base change $i_{x_{\mu}}^! \pi_{\lambda*} \ul{k}_{\widetilde{X_{\lambda}}}$ is isomorphic (up to a shift) to $\pi_{\lambda *}\omega_{F}$. Hence it is enough to show that $\THypgr(\omega_{F})$ is a free $S_k$-module concentrated in even degrees. This is the case by Proposition \ref{prop:BB}(1) above.

For the second statement of the proposition note that the restriction homomorphism $\THypgr(\pi_{\lambda *} \ul{k}_{\widetilde{X_{\lambda}}}) \to \THypgr( (\pi_{\lambda *} \ul{k}_{\widetilde{X_{\lambda}}})_{x_{\mu}})$ is dual (again, up to a shift) to the canonical map $\THypgr(\omega_F) \to \THypgr(\omega_{\widetilde{X_{\lambda}}})$ which is a split injection by Proposition \ref{prop:BB}(2).
\end{proof}

\begin{proof}[Proof of Theorem \ref{thm:parityexists}]
By Proposition \ref{prop:BBres}, $\pi_{\lambda *} \ul{k}_{\widetilde{X_{\lambda}}} \in D_T^b(X,k)$ is parity. If we let $\CQ$ denote an indecomposable summand of $\pi_{\lambda *} \ul{k}_{\widetilde{X_{\lambda}}}$ containing $X_\lambda$ in its support then $\CQ$ is also parity and, by Theorem \ref{thm-parityunique}, we have $\CQ_{X_{\lambda}} \cong \ul{k}_{X_{\lambda}}[i]$ for some integer $i$. It follows that we can take $\CP(\lambda) := \CQ[-i]$.

Another consequence of Theorem \ref{thm-parityunique} is that any indecomposable parity sheaf $\CP(\lambda)$ occurs as a direct summand of $\pi_{\lambda *} \ul{k}_{\widetilde{X_{\lambda}}}[i] \in D_T^b(X,k)$ for some $i$. Hence, to show Part (2) of the theorem it is enough to check that the map
\[
\THypgr(\pi_{\lambda *} \ul{k}_{\widetilde{X_{\lambda}}})
\to \THypgr( (\pi_{\lambda *} \ul{k}_{\widetilde{X_{\lambda}}})_{x_{\mu}}) 
\]
is surjective. This is the case by Proposition \ref{prop:BBres}.

By Proposition \ref{prop:parityselfdual} we have $\DD \CP(\lambda) \cong
  \CP(\lambda)[2 \dim X_{\lambda}]$. We also know that
  $\THypgr(\CP(\lambda))$ is a free $S_k$-module by Proposition
  \ref{prop:BB} (recall that $\omega_{\widetilde{X_{\lambda}}} \cong
  \ul{k}_{\widetilde{X_{\lambda}}}$ because $\widetilde{X_{\lambda}}$
  is smooth). Hence
\[
\Hom^{\bullet}_{S_k} (\THypgr( \CP(\lambda)), S_k) \cong \THypgr(
\CP(\lambda))[2\dim X_{\lambda}].
\]
as $\overline{X_{\lambda}}$ is proper.
\end{proof}

\subsection{Parity sheaves and the functor $\DW$}

In this section we begin discussing the relationship between parity sheaves and the localisation functor $\DW$. In particular, we show that $\DW$  is fully faithful on morphisms of all degrees between parity sheaves.

For the rest of this section we assume (A1)-(A4a/b) and (S1), (S2).

\begin{proposition} \label{prop:globparity} Let $\CP(\lambda)$ be a parity sheaf. Then the localisation homomorphism $\THypgr(\CP(\lambda)) \to \THypgr(\CP(\lambda)_{X^T})$ identifies $\THypgr(\CP(\lambda))$ with the global sections of $\DW(\CP(\lambda))$.
\end{proposition}

\begin{proof} In order to apply Theorem \ref{theorem-loc2} we only need to check that $\THypgr(\CP(\lambda))$ and $\THypgr(\CP(\lambda)_{X^T})$ are free $S_k$-modules. This is the case for $\THypgr(\CP(\lambda))$ by Proposition \ref{prop-propofparity}. For $\THypgr(\CP(\lambda)_{X^T})$ note that because $\CP(\lambda)$ is parity, the restriction of $\CP(\lambda)$ to any $T$-fixed point is a direct sum of equivariant constant sheaves.
\end{proof}

\begin{theorem}
\label{thm:fullyfaithful}
The functor $\DW$ is fully faithful when restricted to
  parity sheaves, i.e.~ if $\CP$ and $\CP^\prime$ are parity sheaves
  on $X$, then
$$
\Hom^\bullet_{D^b_T(X,k)}(\CP,\CP^\prime)\to\Hom_{\CG\catmod^\DZ_k}^\bullet(\DW(\CP),\DW(\CP^\prime))
$$
is an isomorphism.
\end{theorem}

\begin{proof} Without loss of generality we can assume that both $\CP$ and $\CP^\prime$ are even. Let $\lambda\in\Lambda$ be a minimal element and set $\CJ=\Lambda\setminus\{\lambda\}$. Denote by $j\colon X_{\CJ}=X\setminus X_\lambda\to X$ the corresponding open inclusion and by $i\colon X_\lambda\to X$ the complementary closed inclusion. Then we have a distinguished triangle
$$
i_!i^!\CP^\prime\to \CP^\prime\to j_\ast j^\ast\CP^\prime\stackrel{[1]}\to
$$
which gives rise, by Lemma \ref{lemma-ses}, to a short exact sequence
$$
0\to \Hom^\bullet(i^\ast\CP,i^!\CP^\prime)\to\Hom^\bullet(\CP, \CP^\prime)\to\Hom^\bullet(\CP_{\CJ},
\CP^\prime_{\CJ})\to 0
$$
of graded spaces. The map $\Hom^\bullet(\CP, \CP^\prime)\to\Hom^\bullet(\CP_{\CJ},
\CP^\prime_{\CJ})$ is induced by the restriction to an open subspace, hence we can fit the above short exact sequence into a commutative diagram

\centerline{
\xymatrix@C=0.3cm{
0\ar[r]&\Hom(i^\ast\CP,i^!\CP^\prime)\ar[r]\ar[d]&\Hom(\CP, \CP^\prime)\ar[r]\ar[d]&\Hom(\CP_{\CJ},
\CP^\prime_{\CJ})\ar[d] \ar[r]&0\\
0\ar[r]&K\ar[r]&\Hom(\DW(\CP), \DW(\CP^\prime))\ar[r]&\Hom^\bullet(\DW(\CP_{\CJ}),
\DW(\CP^\prime_{\CJ}))\ar[r]&0.
}
}
\noindent
 As $\CP_{\CJ}$ and $\CP^\prime_{\CJ}$ are parity sheaves on $X_{\CJ}$ we can, by induction on the number of strata, assume that the vertical map on the right is an isomorphism. Hence we can finish the proof by showing that the vertical map on the left is an isomorphism as well.

Now $K$ is the space of all morphisms $f\colon \DW(\CP)\to\DW(\CP^\prime)$ with $f^\mu=0$ and $f^E=0$ for vertices $\mu$ and edges $E$ of $\CJ$. By Proposition \ref{prop-propofparity}, (3),  it identifies with the set of homomorphisms from the stalk $\DW(\CP)^\lambda$ into the costalk $\DW(\CP^\prime)_\lambda$.
By definition we have $\DW(\CP)^\lambda=\THypgr(i^\ast \CP)$. Now let us look at the short exact sequence
$$
0\to\THypgr(i^!\CP^\prime)\to\THypgr(\CP^\prime)\to\THypgr(\CP^\prime_{\CJ})\to 0
$$
given by Lemma \ref{lemma-ses}. By Proposition \ref{prop:globparity}, $\THypgr(\CP^\prime)$ and $\THypgr(\CP^\prime_{\CJ})$ can be identified with the sections of $\DW(\CP^\prime)$ over $\Lambda$ and $\CJ$, respectively. Hence we may identify
$$
\THypgr(i^!\CP^\prime)=\DW(\CP^\prime)_\lambda.
$$
As $i^\ast\CP$ and $i^!\CP^\prime$ are free sheaves on $X_\lambda$, we deduce from the above that the homomorphism $\Hom(i^\ast\CP,i^!\CP^\prime)\to K$ in the commutative diagram above is an isomorphism as well.
\end{proof}

\section{Braden-MacPherson sheaves on a moment graph}
We return now to the theory of sheaves on a moment graph. We first
motivate the definition of the Braden-MacPherson sheaves by considering the
problem of extending local sections. Then we prove one of our main results, namely that the functor $\DW$ sends
parity sheaves to Braden-MacPherson sheaves.

\subsection{Extending local sections}

Let $\CG=(\CV,\CE,\alpha)$ be a moment graph. Suppose that  each edge is given  a
direction. Then, for $x,y\in\CV$, we set $x\varle y$ if and only if there is a
directed path from $x$ to $y$.  We assume that this determines a partial order on $\CV$, i.e.~ we assume that there are no non-trivial closed directed paths. We call this datum a {\em directed moment graph}.

Recall that we call a subset $\CJ$ of $\CV$ open if it contains with any element all elements that are larger in the partial order, i.e. all elements that can be reached by a directed path. For a sheaf $\SM$ and an open subset $\CJ$ of $\CV$ we call an element in $\Gamma(\CJ,\SM)$ a {\em local section}.

Now we want to find some conditions on $\SM$ that ensure that each local section can be extended to a global section, i.e. which ensure that the  restriction  $\Gamma(\SM)\to\Gamma(\CJ,\SM)$ from global to local sections is surjective for any open set $\CJ$. For this we need the following definitions.

For a vertex $x$ of $\CG$ we define
$$
\CV_{\delta x}:=\{y\in\CV\mid\text{ there is an edge $E\colon x\to y$}\}.
$$
So $\CV_{\delta x}$ is the subset of vertices $y$ that are bigger than $x$ in the partial order and that are connected to $x$ by an edge. We denote by 
$$
\CE_{\delta x}:=\{E\in \CE\mid E\colon x\to y\}
$$
the set of the corresponding edges. Then there is an obvious correspondence $\CE_{\delta x}\stackrel{\sim}\to\CV_{\delta x}$ (as we assume that two vertices are connected by at most one edge).
 For a sheaf $\SM$ and a vertex $x$ we define the map
$$
u_x\colon \Gamma(\{>x\},\SM)\to\bigoplus_{E\in\CE_{\delta x}}\SM^E
$$
as the composition
$$
\Gamma(\{>x\},\SM)\subset\bigoplus_{y>x}\SM^y \stackrel{p}\to\bigoplus_{y\in\CV_{\delta x}} \SM^y\stackrel{\rho}\to\bigoplus_{E\in\CE_{\delta x}}\SM^E,
$$
where $p$ is the projection along the decomposition and $\rho=\bigoplus_{y\in\CV_{\delta x}}\rho_{y,E}$. We let
$$
\SM^{\delta x}:=u_x(\Gamma(\{>x\},\SM))\subset\bigoplus_{E\in\CE_{\delta x}}\SM^E
$$
be the image of this map. Moreover, we define the map
$$
d_x:=(\rho_{x,E})_{E\in\CE_{\delta x}}^T\colon \SM^x\to\bigoplus_{E\in\CE_{\delta x}}\SM^E.
$$
The connection of the above definitions with the problem of extending local sections is the following.
Suppose that $m^\prime\in \Gamma(\{>x\},\SM)$ is a section and that $m_x\in \SM^x$. Then the concatenated element $(m_x,m^\prime)\in \bigoplus_{y\ge x}\SM^y$ is a section over $\{\ge x\}$ if and only if $u_x(m^\prime)=d_x(m_x)$.

\begin{lemma}\label{lemma-extension} For a sheaf $\SM$ on the moment graph $\CG$ the following are equivalent:
\begin{enumerate}
\item For any open subsets $\CJ^\prime\subset\CJ$ of $\CV$ the restriction map $\Gamma(\CJ,\SM)\to\Gamma(\CJ^\prime,\SM)$ is surjective.
\item For any vertex $x\in\CV$, the restriction map $\Gamma(\{\ge x\},\SM)\to \Gamma(\{>x\},\SM)$ is surjective.
\item For any $x\in\CV$, the map $d_x\colon \SM^x\to \bigoplus_{E\in\CE_{\delta x}} \SM^E$ contains $\SM^{\delta_x}$ in its image.
\end{enumerate}
\end{lemma}
\begin{proof} Clearly property (2) is a special case of property (1). Let us prove the converse, so let us assume that (2) holds. It is enough to prove property (1) in the special case that $\CJ=\CJ^\prime\cup\{x\}$ for a single element $x$, since we get the general case from this by induction. So let $m=(m_y)$ be a section in $\Gamma(\CJ\setminus\{x\}, \SM)$. Since $\{>x\}\subset \CJ\setminus\{x\}$ we can restrict $m$ and get a section $m^\prime$ in $\Gamma(\{>x\},\SM)$. By assumption there is an element $m_x\in \SM^x$ such that $(m_x,m^\prime)$ is a section over $\{\ge x\}$. As $x$ is not connected to any vertex in $\CJ\setminus\{\ge x\}$ it follows  that $(m_x, m)$ is a section over $\CJ$. Hence (2) implies (1).

Let us show that (2) is equivalent to (3). Now (2) means that for any section $m$ over $\{>x\}$ we can find $m_x\in \SM^x$ such that $(m_x,m)$ is a section over $\{\ge x\}$. But $(m_x,m)$ is a section over $\{\ge x\}$ if and only if $d_x(m_x)=u_x(m)$. Hence, a section $m$ over $\{>x\}$ can be extended to the vertex $x$ if and only of $u_x(m)$ is contained in the image of $d_x$.
\end{proof}

For later use we prove the following statement.

\begin{lemma} \label{lemma-suronstalks} Let $x$ be a vertex of $\CG$ and $\SM$ a sheaf on $\CG$. Then the following are equivalent:
\begin{enumerate}
\item The map $\Gamma(\{\ge x\},\SM)\to \SM^x$ is surjective.
\item The image of the map $d_x\colon \SM^x\to \bigoplus_{E\in\CE_{\delta x}}\SM^E$ is contained in $\SM^{\delta x}$.
\end{enumerate}
\end{lemma}

\begin{proof} Suppose that (1) holds and let $s\in \SM^x$. Then there is a section $m$ of $\SM$ over $\{\ge x\}$ with $m_x=s$. If we denote the restriction of $m$ to $\{>x\}$ by $m^\prime$, then this means that $d_x(s)=u_x(m^\prime)$. So $d_x(s)$ is contained in the image of $u_x$, which is $\SM^{\delta x}$.

Conversely, suppose that (2) holds and let $s\in \SM^x$. Then there is a section $m^\prime$ of $\SM$ over $\{>x\}$ such that $d_x(s)=u_x(m^\prime)$. Hence $(s,m^\prime)$ is a section over $\{\ge x\}$, hence $s$ is contained in the image of $\Gamma(\{\ge x\},\SM)\to \SM^x$.
\end{proof}

\subsection{Braden--MacPherson sheaves}
The most important class of sheaves on a moment graph is the following.
\begin{definition} \label{def-BMPsheaves} A sheaf $\SB$ on the moment graph $\CG$ is called a {\em Braden--MacPherson sheaf} if it satisfies the following properties:
\begin{enumerate}
\item For any $x\in\CV$, the stalk $\SB^x$ is a graded free $S_k$-module of finite rank and only finitely many $\SB^x$ are non-zero. 
\item For a directed edge $E\colon x\to y$ the map $\rho_{y,E}\colon \SB^y\to\SB^E$ is surjective with kernel $\alpha(E)\SB^y$,
\item For any open subset $\CJ$ of $\CV$ the map $\Gamma(\SB)\to \Gamma(\CJ,\SB)$ is surjective.
\item The map $\Gamma(\SB)\to \SB^x$ is surjective for any $x\in \CV$.
\end{enumerate}
\end{definition}

Here is a classification theorem.

\begin{theorem} Assume that $k$ is a local ring and suppose that the moment graph is such that for any vertex $w$ the set $\{\le w\}$ is finite. Then the following holds.
\begin{enumerate}
\item For any vertex $w$ there is an up to isomorphism unique Braden--MacPherson sheaf $\SB(w)$ on $\CG$ with the following properties:
\begin{itemize}
\item We have $\SB(w)^w\cong S_k$ and $\SB(w)^x= 0$ unless $x\le w$.
\item $\SB(w)$ is indecomposable in $\CG\catmod_k^\DZ$.
\end{itemize}
\item Let $\SB$ be a Braden--MacPherson sheaf. Then there are $w_1,\dots,w_n\in \CV$ and $l_1,\dots,l_n\in\DZ$ such that
$$
\SB\cong\SB(w_1)[l_1]\oplus \dots\oplus\SB(w_n)[l_n].
$$
The multiset $(w_1,l_1)$,\dots,$(w_n,l_n)$ is uniquely determined by $\SB$.
\end{enumerate}
\end{theorem}
\begin{remark} We need the locality assumption on $k$ in order to ensure that projective covers exist in the category of graded $S_k$-modules.
\end{remark}

\begin{proof} We first prove the existence part of statement (1). For $w\in\CV$ we define a sheaf $\SB(w)$ by the following inductive construction:
\begin{enumerate}
\item We start with setting $\SB(w)^w=S_k$ and $\SB(w)^x=0$ if $x\not\le w$.
\item If we have already defined $\SB(w)^y$, then we set, for each edge $E\colon x\to y$, 
$$
\SB(w)^E:=\SB(w)^y/\alpha(E) \SB(w)^y
$$
and we let $\rho_{y,E}\colon\SB(w)^y \to \SB(w)^E$ be the canonical map.
\item Suppose that we have already defined $\SB(w)^y$ for all $y$ in an open subset $\CJ$ and suppose that $x\in\CV$ is such that $\CJ\cup\{x\}$ is open as well. By step (2) we have also defined the spaces $\SB(w)^E$ for each edge $E\colon x\to y$ originating at $x$, as well as the maps $\rho_{y,E}\colon \SB(w)^y\to \SB(w)^E$. We now define $\SB(w)^x$ and the maps $\rho_{x,E}$ for those edges $E$. We can already calculate the sections of $\SB(w)$ over $\{>x\}$, as well as the $S_k$-modules $\SB(w)^{\delta x}\subset \bigoplus_{E\in\CE_{\delta x}}\SB(w)^E$. Now we define $d_x\colon \SB(w)^x\to \SB(w)^{\delta x}$ as a projective cover  in the category of graded $S_k$-modules. The components of $d_x$ (with respect to the inclusion $\SB(w)^{\delta x}\subset \bigoplus_{E\in\CE_{\delta x}}\SB(w)^E$) give us the maps $\rho_{x,E}$.
\end{enumerate}

Let us check that $\SB(w)$ is indeed a Braden--MacPherson sheaf. Since $\SB(w)^x$ is projective for all $x\in\CV$ it is a graded free $S_k$-module and the finiteness assumptions hold as well, so $\SB(w)$ fulfills property (1). Property (2)  is assured by step (2) in the inductive construction of $\SB(w)$. Property (3) is, by Lemma \ref{lemma-extension}, equivalent to the fact that for all $x\in\CV$ the map $d_x\colon \SB(w)^x\to \bigoplus_{E\in\CE_{\delta x}} \SB(w)^E$ contains $\SB(w)^{\delta_x}$ in its image. This is clear by step (3).
In addition, step (3) also yields that the image of $d_x$ is contained in $\SB(w)^{\delta x}$. By Lemma \ref{lemma-suronstalks} this is equivalent to the surjectivity of $\Gamma(\{\ge x\},\SB(w))\to\SB(w)^x$. We have already seen that the restriction map $\Gamma(\SB(w))\to\Gamma(\{\ge x\},\SB(w))$ is surjective. Hence also $\Gamma(\SB(w))\to\SB(w)^x$ is surjective, hence $\SB(w)$ also satisfies property (4) of a Braden--MacPherson sheaf.

Now we prove statement (2) of the above theorem using the above explicitly defined objects $\SB(w)$. Note that this also gives the uniqueness part of statement (1), which we have not yet proven. So let $\SB$ be an arbitrary Braden--MacPherson sheaf. We prove by induction on the set of open subsets $\CJ$ of $\CV$ that there are $(w_1,l_1)$,\dots,$(w_n,l_n)$ such that
$$
\SB^\CJ\cong \SB(w_1)^\CJ[l_1]\oplus \dots\oplus\SB(w_n)^\CJ[l_n].
$$
(Here and in the following we denote by $\SF^\CI$ the obvious restriction of a sheaf $\SF$ to the subgraph corresponding to the vertices in a subset $\CI$ of $\CV$.) 

So suppose that $\CJ$ is open, that $x\in\CJ$ is minimal and we have a decomposition as above for $\CJ^\prime=\CJ\setminus \{x\}$. We get, in particular,
$$
\Gamma(\{>x\},\SB)\cong\Gamma(\{>x\},\SB(w_1)[l_1])\oplus \dots\oplus\Gamma(\{>x\},\SB(w_n)[l_n])
$$
 and
$$
\SB^{\delta x}\cong \SB(w_1)^{\delta x}[l_1]\oplus \dots\oplus\SB(w_n)^{\delta x}[l_n].
$$
Now $d_x\colon\SB^x\to \SB^{\delta x}$ is surjective, by property (3) of a Braden--MacPherson sheaf and Lemma \ref{lemma-extension}, and $\bigoplus \SB(w_i)^x[l_i]\to \bigoplus\SB(w_i)^{\delta x}[l_i]$ is a projective cover by construction. Hence we have a decomposition $\SB^x=\bigoplus \SB(w_i)^x[l_i]\oplus R$ for some graded free $S_k$-module $R$ which lies in the kernel of $d_x$. Each isomorphism $R\cong S_k[m_1]\oplus\dots\oplus S_k[m_r]$ then yields an isomorphism
$$
\SB^\CJ\cong \SB(w_1)^\CJ[l_1]\oplus \dots\oplus\SB(w_n)^\CJ[l_n]\oplus \SB(x)^{\CJ}[m_1]\oplus\dots\oplus \SB(x)^{\CJ}[m_r],
$$
which is our claim for $\CJ$.
The above construction also yields the uniqueness of the multiset $(w_1,l_1)$,\dots,$(w_n,l_n)$.
\end{proof}

\subsection{Directed moment graphs from stratified varieties}
Suppose that $X$ is a complex $T$-variety satisfying (A1). In Section \ref{subsec-mgfromvar} we constructed an (undirected) moment graph $\CG_X$ from this datum. Suppose now that, in addition, we are given a stratification $X=\bigsqcup_{\lambda\in\Lambda} X_\lambda$ satisfying (S1) and (S2). Recall that for each $\lambda\in\Lambda$ we denote by $x_\lambda$ the  unique fixed point in $X_\lambda$. Hence we now have identifications between the set of fixed points in $X$, the set $\Lambda$ of strata and the set of vertices of $\CG_X$. 

From this we obtain a direction of each edge as follows. Suppose that
the one-dimensional orbit $E$ contains $x_\lambda$ and $x_\mu$ in its
closure. Then either $X_\lambda\subset \ol{X_\mu}$ or
$X_\mu\subset\ol{X_\lambda}$. We direct the corresponding edge of
$\CG_X$ towards $\mu$ in the first case, and towards $\lambda$ in
the second case. We denote by $\le$ the partial order on the vertices of $\CG_X$
generated by the relation $\lambda \le \mu$ if there is an edge $E :
\lambda \to \mu$. The following proposition shows that this is the
same order as the one induced by the closure relations on the strata:

\begin{proposition} We have $\lambda \le \mu$ if and only if $X_{\lambda}
  \subset \overline{X_{\mu}}$. \end{proposition}

\begin{proof}
Clearly, if $\lambda \le \mu$ then $X_{\lambda}
  \subset \overline{X_{\mu}}$. For the converse we show:
\begin{equation} \label{eq:claim}
\begin{array}{c} \text{
If $X_{\lambda} \subset \overline{X_{\mu}}$ then there exists an
edge}\\
\text{$E : \lambda \to \nu$ such that $X_{\nu} \subset
\overline{X_{\mu}}$.}
\end{array}
\end{equation}
Let us assume that \eqref{eq:claim} holds. Then, if $X_{\lambda}
\subset \overline{X_{\mu}}$, we can find a chain
$\lambda \to \nu_1 \to \dots \to \nu_m \to \mu$ and so $\lambda \le \mu$.

It remains to show \eqref{eq:claim}. Let $U$ be an affine
$T$-stable neighbourhood of $x_{\lambda}$ in $\overline{X_{\mu}}$ 
and let $N_{\lambda} \subset
U$ be a $T$-invariant normal slice to the stratum $X_{\lambda}$ at
$x_{\lambda}$. Because $x_{\lambda} \in
N_{\lambda}$ is attractive, we can find a cocharacter $\gamma:
\mathbb{C}^\times \to T$ such that $\lim_{z \to 0} \gamma(z) \cdot x =
x_{\lambda}$ for all $x \in N_{\lambda}$. It follows that $Y := (N_{\lambda} \setminus \{ x_\lambda \}) /
\mathbb{C}^\times$  is a projective variety. By 
Borel's fixed point theorem, $T$ has a fixed point on $Y$ and
hence a one-dimensional orbit on $N_{\lambda}$. This one-dimensional
orbit is contained in some $X_{\nu}$, hence connects $x_\lambda$ with $x_\nu$. By construction we have
$X_{\nu} \subset \overline{X_{\mu}}$.
\end{proof}

\subsection{The $k$-smooth locus of a moment graph} \label{subsec:ksm}

In this subsection we assume that $k$ is a field and that the (directed) moment graph $\CG$ contains a largest vertex $w$. This moment graph carries the indecomposable Braden--MacPherson sheaf $\SB:=\SB(w)$ over $k$. 

\begin{definition} The {\em $k$-smooth locus} $\Omega_k(\CG)$ of $\CG$ is the set of all vertices $y$ of $\CG$ such that $\SB^y$ is a free $S_k$-module of (ungraded) rank $1$.
\end{definition}

In \cite{Fie06} the $k$-smooth locus is  determined for a large class of pairs $(\CG,k)$. In order to formulate the result, let $\CB:=\Gamma(\SB)$ be the space of global sections of $\SB$. We consider this as a graded $\CZ_k$-module. 

\begin{definition} We say that $\CB$ is self-dual of degree $l\in\DZ$ if there is an isomorphism
$$
\Hom^\bullet_{S_k}(\CB,S_k)\cong\CB[l]
$$
of graded $\CZ_k$-modules.
\end{definition}

The following is an analogue of the assumption (A4a) for moment graphs.
\begin{definition} We say that the pair $(\CG,k)$ is a {\em GKM-pair} if $\alpha_E$ is non-zero in $Y\otimes_\DZ k$ for any edge $E$ and if  for any distinct edges $E$ and $E^\prime$ containing a common vertex we have $k\alpha_E\cap k\alpha_{E^\prime}=0$. 
\end{definition}
Note that this can be considered, for given $\CG$, as a requirement on the characteristic of $k$. The main result of \cite{Fie06} is the following:

\begin{theorem}[{\cite[Theorem 5.1]{Fie06}}] \label{thm:mgsmoothlocus}
Suppose that $(\CG,k)$ is a GKM-pair and that $\CB$ is self-dual of degree $2l$. Then we have
$$
\Omega_k(\CG)=\left\{x\in\CV\left| \begin{matrix} \text{ for all $y\ge x$ the number}\\ \text{ of edges containing $y$ is $l$}
\end{matrix}\right\}\right..
$$
\end{theorem}

We are going to apply this statement later in order to study the $k$-smooth locus of $T$-varieties.

\subsection{The combinatorics of parity sheaves}
Let $X$ be a complex $T$-variety, $k$ a complete local principal ideal domain. Assume that these data satisfy the assumptions (A1)--(A4a/b), (S1), (S2) and (R1)--(R3). We now come to the principal result of this paper. 

\begin{theorem} \label{thm:Wparity=BM}
Suppose that $\CP\in D_T^b(X,k)$ is a parity sheaf. Then $\DW(\CP)$ is
a Braden--MacPherson sheaf. More precisely, $\DW(\CP(\lambda)) \cong \SB(\lambda)$.
\end{theorem}
\begin{proof} We have to show that $\DW(\CP)$ satisfies the four properties listed in Definition \ref{def-BMPsheaves}. If we translate this definition into our situation we see that we have to check the following:
\begin{enumerate}
\item
 For each $x\in X^T$ the cohomology $\THypgr(\CP_x)$ is a graded free module over $S_k$.
\item For each one-dimensional orbit $E$ that is contained in the stratum associated to the fixed point $x$, the map
$$
\THypgr(\CP_x)\to\THypgr(\CP_E)
$$
is surjective with kernel $\alpha(E)\THypgr(\CP_x)$.
\item For each open union $X_\CJ\subset X$ of strata the restriction homomorphism
$$
\THypgr(\CP)\to \THypgr(\CP_{\CJ})
$$
is surjective.
\item For each $x\in X^T$ the homomorphism $\THypgr(\CP)\to \THypgr(\CP_x)$ is surjective.
\end{enumerate}
Part (1) follows directly from the definition of a parity sheaf, the parts
(2) and (3) are stated in Proposition \ref{prop-propofparity}. Part (4) follows from Theorem
\ref{thm:parityexists} and the fact $\CP$ is the direct sum of shifted
copies of $\CP(\lambda)$'s. The last statement follows, as
$\CP(\lambda)$ is indecomposable if and only if $\DW(\SB(\lambda))$ is,
by Theorem \ref{thm:fullyfaithful}.
\end{proof}

\section{The case of Schubert varieties}

We now discuss a special and important case of the general theory
developed in the previous section, namely the case of Schubert
varieties in (Kac-Moody) flag varieties. For a detailed construction
of these varieties in the Kac-Moody setting the reader is referred to
\cite{Ku}.

We fix some notation. Let $A$ be a generalised Cartan matrix of size
$l$ and let $\fg=\fg(A) = \fn_- \oplus \fh \oplus \fn_+$ denote the corresponding Kac-Moody Lie algebra
with Weyl group $W$, Bruhat order $\le$, length function $\ell$ and
simple reflections $S = \{ s_i\}_{i=1, \dots, l}$. To $A$ one may also
associate a Kac-Moody group $G$ with Borel subgroup $B$ and connected algebraic torus $T\subset B$.  Given any subset $I \subset \{1,
\dots, l \}$ one has a standard parabolic subgroup $P_I$ containing
$B$ and standard parabolic subgroup $W_I \subset W$.  The set $G / P_I$ may be given the structure of an
ind-$T$-variety and is called a Kac-Moody flag variety.

For each $w \in W$ one may consider the \emph{Schubert cell} $X_w^I:=B
w P_I / P_I  \subset G / P_I$ and its closure, the
\emph{Schubert variety},
\[
\ol{X_w^I}  = \bigsqcup_{v \le w} B v P_I.
\]
Each Schubert cell   is isomorphic to a (finite dimensional) affine space and each Schubert
variety is a (finite dimensional) projective algebraic variety. The partition of $G
/ P_I$ into Schubert cells gives a stratification of $G
/ P_I$.

The following proposition shows that the results of this article may
be applied to any closed union of finitely many $B$-orbits in $G / P_I$:

\begin{proposition} Let $X \subset G
/ P_I$ be a closed subset which is the union of finitely many
Schubert cells. Then $X$ together with its stratification into
Schubert cells satisfies our assumptions (A1), (A2), (S1), (S2), (R1),
(R2) and (R3).
\end{proposition}

\begin{proof} The assumptions (A1), (A2), (S1) are standard properties of Kac-Moody
  Schubert varieties (see \cite[Chapter 7]{Ku}) and (S2) follows
  because we have an equivalence $D_{T,\Lambda}^b(X,k) \cong
  D_{B}^b(X,k)$. Given a Schubert variety $\ol{X_w^I} \subset
  X$, let $\pi: \widetilde{X} \to \ol{X^I_w}$ denote a Bott-Samelson resolution
 (see \cite[7.1.3]{Ku}). Then $\widetilde{X}$ is a smooth
  ${T}$-variety with  finitely many ${T}$-fixed points
  which admits a $T$-equivariant closed linear embedding into a
  projective space. Lastly,
  the variety $\widetilde{X}$ is even a $B$-variety, and the map
  $\pi$ is $B$-equivariant. So properties (R1), (R2) and (R3) hold as well.
\end{proof}

We now describe the moment graph of $G/ P_I$. The identification
of $\fh$ with the Lie algebra of $T$ allows us to identify the
lattice of characters $X^\ast(T)$ with a
lattice in $\fh^*$. Moreover, under this identification, all the roots of $\fg(A)$
 lie in $X^\ast(T)$. Let $R \subset X^\ast(T)$ denote the subset
of real roots, and $R^+$ the subset of real positive
roots. Then we have a bijection
\begin{eqnarray*}
R^+ \stackrel{\sim}{\to} & \{ \text{reflections in $W$} \} \\
\alpha \mapsto & s_{\alpha}.
\end{eqnarray*}
The following proposition follows from \cite[Chapter 7]{Ku}:

\begin{proposition} \label{prop:moment graph for G/B}
We have:
\begin{itemize}
\item The $T$-fixed points are in bijection with the set $W/W_I$:
\begin{gather*}
W/W_I \stackrel{\sim}{\to} ( G / P_I)^{T} \\
wW_I \mapsto wP_I.
\end{gather*}
\item There is a one-dimensional $T$-orbit with $xW_I$ and
  $yW_I$ in its closure if and only if $s_{\alpha} xW_I = yW_I$ for
  some reflection $s_{\alpha} \in W$ in
  which case $T$ acts on this orbit with character $\pm
  \alpha$.
\end{itemize}
\end{proposition}

We complete this section by discussing what the arithmetic assumptions
(A3), (A4a) and (A4b) mean in the case of Kac-Moody flag
varieties. First note that the lattice $\DZ R \subset X^\ast(T)$ spanned
by the real roots determines a surjection of algebraic tori
\[
s: T \twoheadrightarrow T'
\]
so that $X^\ast(T') = \DZ R$. The action of $T$ on $G / P_I$ is
trivial on the kernel of $s$ and we obtain an action of $T'$ on
$G/ P_I$. (In the case of a finite dimensional Schubert variety
this corresponds to the fact that one may always choose the adjoint
form of a reductive group in order to construct the flag variety.)
Because real roots are never divisible in $\DZ R = X^\ast(T')$ it follows that
(A4b) (and hence (A3)) is always satisfied for Kac-Moody Schubert
varieties viewed as $T'$-varieties.

The condition (A3) is more subtle. If we fix a field $k$ then
condition (A3) is satisfied if and only if no two distinct  roots in $R^+$
become linearly dependent modulo $k$. One may check that in the finite cases we have to exclude characteristic $2$ in non-simply laced types and characteristic $3$ in type $G_2$.

In the affine case the situation is radically different. Suppose that
$\widehat G$ is the affine Kac-Moody group associated to a semi-simple group
$G$. Recall that there exists a an element $\delta \in \widehat\fh^\ast$ such that
the set of real roots of $\widehat G$ is equal to $\{ \alpha + n\delta \}$
where $\alpha \in \fh^\ast$ is a root of $G$, and $n \in \DZ$. It follows
that condition (A3) is satisfied for $\widehat G / \widehat P$ and any parabolic
subgroup $\widehat P \ne \widehat G$ if and only if $k$ is of characteristic $0$.

However, if one restricts oneself to a fixed a Schubert variety $X \subset
\widehat G / \widehat P$ the GKM-condition for $X$ may yield interesting restrictions on
the characteristic of $k$ (see \cite{Fie3}).

\section{$p$-Smoothness}

In this section we recall the definition and basic properties of the
$p$-smooth locus of a complex algebraic variety $X$. Our main goal is
Theorem \ref{thm:psmoothmaintheorem} for which we need Proposition
\ref{prop:icstalks+psmooth}, where we show that an (a priori weaker)
condition on the stalks of the intersection cohomology complex is
enough to conclude $p$-smoothness.

Throughout this section all varieties will be irreducible and $k$
denotes a ring (assumed to be a field from Sections \ref{sec:ksmoothIC} to \ref{subsec:pT}).
Dimension will always refer to the complex dimension. Given a point $y$ in a variety $Y$ 
 we denote by $i_y : \{y \} \hookrightarrow Y$ its inclusion.

\subsection{Smoothness and $p$-smoothness} 

If $x$ is a smooth point of a variety $X$ of dimension $n$ a simple
calculation (using the long exact sequence of cohomology, excision and
the cohomology of a $2n-1$-sphere) yields
\[
H^d(X, X \setminus \{ x \}, k) =
\begin{cases} k,\quad \text{if $d = 2 \dim X$,} \\
0, \quad \text{otherwise.}
\end{cases} \]
The isomorphism
\begin{equation} \label{eq:closedsup}
H^d(i_x^! \underline{k}_{X}) \cong H^d(X, X \setminus \{ x \}, k)
\end{equation}
motivates the following.

\begin{definition} A variety $X$ is
\emph{$k$-smooth} if, for all $x \in X$, one has an isomorphism
\[
H^d(i_x^! \underline{k}_X) \cong
\begin{cases} k, \quad \text{if $d = 2 \dim X$}, \\
0, \quad \text{otherwise.}
\end{cases} \]
The \emph{$k$-smooth locus} of $X$ is the largest open $k$-smooth
subvariety of $X$. We define \emph{$p$-smooth} (respectively the
\emph{$p$-smooth locus}) to mean $\mathbb{F}_p$-smooth (respectively the
\emph{$\mathbb{F}_p$-smooth locus}).
\end{definition}

\begin{proposition} We have inclusions
\begin{equation*}
\begin{array}{c} \DQ \text{-smooth} \\ \text{locus} \end{array} æ\supset
\begin{array}{c} p \text{-smooth} \\ \text{locus} \end{array} æ\supset
\begin{array}{c} \DZ \text{-smooth} \\ \text{locus} \end{array} æ\supset
\begin{array}{c} \text{smooth} \\ \text{locus} \end{array}.
\end{equation*}
\end{proposition}

\begin{proof} The fact that the $\mathbb{Z}$-smooth locus contains the
smooth locus follows from the above discussion.
For all rings $k$ one has an isomorphism
\begin{equation} \label{eq:redZ}
i_x^! \underline{k}_X \cong i_x^! \underline{\mathbb{Z}}_X
\stackrel{L}{\otimes}_{\mathbb{Z}} k.
\end{equation}
As the category of $\mathbb{Z}$-modules is hereditary, every object in
$D^b(\{x \}, \mathbb{Z})$ is isomorphic to its cohomology. It then
follows from \eqref{eq:redZ} that, for a field $k$, the condition of
$k$-smoothness is satisfied at $x$ if and only if:
\begin{enumerate}
\item $H^d(i_x^! \underline{\mathbb{Z}}_X)$ is torsion except for $d =
2n$, where the free part is of rank 1,
\item all torsion in $H^\bullet(i_x^! \underline{\mathbb{Z}}_X)$ is prime to
the characteristic of $k$.
\end{enumerate}
The claimed inclusions now follow easily.
\end{proof}

\begin{remark} The above proof shows that, if $k$ is a field, then the
$k$-smooth locus only depends on the characteristic of $k$.
\end{remark}

\subsection{$k$-smoothness and the intersection cohomology complex}
\label{sec:ksmoothIC}
Until Section \ref{subsec:pT} we assume that $k$ is a field. 
Let us denote by $(D^{\le 0}(X,k), D^{\ge 0}(X,k))$ the standard
$t$-structure on $D(X,k)$ with heart $Sh(X,k)$, the abelian category of
sheaves of $k$-vector spaces on $X$. We denote the corresponding truncation
and cohomology functors by $\tau_{\le 0}$, $\tau_{>0}$ and
$\mathcal{H}^d$.

Let us a fix a Whitney stratification  $X = \bigsqcup_{\lambda \in
\Lambda} X_{\lambda}$ and denote for all $\lambda \in \Lambda$ by
$i_{\lambda} \colon X_{\lambda} \hookrightarrow X$ the inclusion. Recall
that the intersection cohomology complex of $X$, $\IC(X,k) \in
D(X,k)$, is uniquely determined by the properties:
\begin{enumerate}
\item $i_{\lambda}^* \IC(X,k) \cong \underline{k}_{X_{\lambda}}$ for
the open stratum $X_{\lambda} \subset X$;
\end{enumerate}
and, for all strata $X_{\lambda}$ of strictly positive codimension,
\begin{enumerate}
\item[(2)] $\mathcal{H}^d (i_{\lambda}^* \IC(X,k)) = 0$ for $d \ge
\codim_X X_{\lambda}$,
\item[(3)] $\mathcal{H}^d (i_{\lambda}^! \IC(X,k)) = 0$ for $d \le
\codim_X X_{\lambda}.$
\end{enumerate}
Note that under this normalisation $\IC(X,k)$ is not Verdier self-dual.
Rather $\mathbb{D} \IC(X,k) \cong \IC(X,k)[ 2 \dim X]$.
Conditions (2) and (3) are equivalent to the conditions
\begin{enumerate}
\item[(2*)] $\mathcal{H}^d(i_x^* \IC(X,k)) = 0$ for $d \ge \codim_X X_{\lambda}$,
\item[(3*)] $\mathcal{H}^d(i_x^! \IC(X,k)) = 0$ for $d \le \dim X + \dim X_{\lambda}$
\end{enumerate}
for all $x \in X_{\lambda}$ and strata $X_{\lambda}$ of strictly
positive codimension.
(This follows from the fact if $y \in Y$ is a smooth point, then one
has an isomorphism $i_y^! \underline{k}_{X_\lambda} \cong i_y^*\underline{k}_{X_\lambda}[ -2 \dim Y]$.)

\begin{proposition}
\label{prop:psmoothconstant}
A variety $X$ is $k$-smooth if and only if $\IC(X,k) \cong \underline{k}_X$.
\end{proposition}

\begin{proof} If $X$ is $k$-smooth then the constant sheaf  $\underline{k}_X$ satisfies
(1), (2*) and (3*) above and hence $\underline{k}_X \cong \IC(X,k)$. On
the other hand, if $\underline{k}_X \cong \IC(X,k)$ then $\mathbb{D}
\underline{k}_X \cong \underline{k}_X[2 \dim X]$ and
for all $x \in X$ we have
\[ i_x^! \underline{k}_X \cong
i_x^! (\mathbb{D} \underline{k}_X )[-2 \dim X] \cong \mathbb{D} (i_x^*
\underline{k}_X [2 \dim X]) \]
and hence
\[ H^d(i_x^! \underline{k}_X) = H^{-d}(i_x^* \underline{k}_X[ 2\dim X])
= \begin{cases} k, \quad \text{if $d = 2\dim X$,} \\
0, \quad \text{otherwise,} \end{cases}\]
and so $x$ is $k$-smooth. \end{proof}

\subsection{$k$-smoothness and stalks}

Given a morphism $f: X \to Y$ of complex algebraic varieties we write
$^0f_*$ for the non-derived direct image functor. The functor $^0f_*$ is
left $t$-exact with respect to the standard $t$-structure. Given $\CF
\in Sh(X,k)$ we have $^0f_* \CF \cong \tau_{\le 0} f_* \CF$ canonically.

\begin{lemma} Given $\CF \in D^{\ge 0}(X,k)$ and a morphism $f
: X \to Y$ we have a natural isomorphism $\tau_{\le 0} f_* \CF \cong æ{}^0f_*
\tau_{\le 0} \CF$. \end{lemma}

\begin{proof} Applying $f$ to the distinguished triangle
\[ \tau_{\le 0} \CF \to \CF \to \tau_{> 0} \CF \to \]
yields a distinguished triangle
\[ f_* \tau_{\le 0} \CF \to f_* \CF \to f_* \tau_{> 0} \CF \to .\]
Now $f_*$ is left $t$-exact for the $t$-structure $(D^{\le 0}(X,k),
D^{\ge 0}(X,k))$ and so $\tau_{\le 0} f_* \tau_{> 0} \CF = 0$. Hence
if we apply $\tau_{\le 0}$ to the above distinguished triangle we
obtain the required isomorphism
\[
{}^0 f_* \tau_{\le 0} \CF = \tau_{\le 0} f_* \tau_{\le 0} \CF
\stackrel{\sim}{\to} \tau_{\le 0} f_* \CF. \qedhere \]
\end{proof}

\begin{lemma} We have an isomorphism $\tau_{\le 0} \IC(X,k) \cong {}^0
j_* \underline{k}_U$, where $j : U \hookrightarrow X$ denotes the
inclusion of a smooth, open, dense subvariety of $X$. \end{lemma}

\begin{proof} Choose a stratification of $X$ which has $U$ as the only
stratum of dimension $n$, and write $X_i$ for the union of all strata
of codimension $\le i$ (so that $X_0 = U$ and $X_n = X$). We have a
chain of inclusions
\[
X_0 \stackrel{j_0}{\hookrightarrow}
X_1 \stackrel{j_1}{\hookrightarrow}
X_2 \stackrel{j_0}{\hookrightarrow} \dots
\stackrel{j_{n-2}}{\hookrightarrow} X_{n-1}
\stackrel{j_{n-1}}{\hookrightarrow} X_n.
\]
The Deligne construction (see \cite[Proposition 2.1.11]{BBD}) gives an isomorphism
\[
\IC(X,k) \cong
(\tau_{\le n-1} \circ j_{{n-1}*} ) \circ
(\tau_{\le n-1} \circ j_{{n-1}*} ) \circ
\dots
(\tau_{\le 0} \circ j_{{0}*} \underline{k}_U)
\]
Repeatedly applying the above lemma yields
\[
\tau_{\le 0} \IC(X,k) \cong
{}^0 j_{{n-1}*} \circ \dots \circ {}^0 j_{{1}*} \circ {}^0 j_{0*}
\underline{k}_U
\cong {}^0 j_* \underline{k}_U,
\]
as claimed.
\end{proof}

\begin{proposition} \label{prop:icstalks+psmooth}
A variety $X$ is $k$-smooth if and only if $\IC(X,k)_x \cong k$ for all $x
\in X$. \end{proposition}

\begin{proof} If $X$ is $k$-smooth then $\IC(X,k) \cong \underline{k}_X$
by Proposition \ref{prop:psmoothconstant} and so $\IC(X,k)_x \cong k$
for all $x \in X$.
It remains to show the converse. Choose an open, dense, smooth
subvariety $U$ of $X$ and let $j : U \hookrightarrow X$ denote its
inclusion.
The adjunction morphism
\[ \ul{k}_{X} \to {}^0j_* j^* \underline{k}_X æ\]
is an injection on stalks, as may easily be checked from the
definition of ${}^0j_*$. (It is an isomorphism if and only if $X$ is
unibranched, however we will not need this fact.) It follows from our assumptions that
$\IC(X,k)$ lies in $D^{\le 0}(X,k)$ and so we have an isomorphism
\[ \tau_{\le 0} \IC(X,k) \stackrel{\sim}{\to} \IC(X,k). \]
By the above lemma we also have an isomorphism $ {}^0 j_* \underline{k}_U\cong \tau_{\le 0}\IC(X,k)$. It follows that all stalks of
${}^0 j_* \underline{k}_U$ are one-dimensional and that we have an
isomorphism
\[
\underline{k}_X \stackrel{\sim}{\to} {}^0 j_* \underline{k}_U
\stackrel{\sim}{\to} \IC(X,k).\]
Our claim now follows from Proposition  \ref{prop:psmoothconstant}. \end{proof}

\subsection{On the $p$-smooth locus of $T$-varieties} \label{subsec:pT}

Now let $X=\bigsqcup_{\lambda\in\Lambda} X_\lambda$ be an irreducible,
complex, stratified $T$-variety, and let $k$ be a field. Assume that
these data satisfy the assumptions (A1)--(A4a/b), (S1), (S2) and
(R1)--(R3) and let $\Omega_k(\CG)$ denote the $k$-smooth locus of the
moment graph $\CG$ of $X$. The following proposition shows that the
$k$-smooth locus of $X$ and of its moment graph agree.

\begin{proposition} All points of a stratum $X_{\lambda}$ belong to the $k$-smooth locus of $X$ if and only if $\lambda \in \Omega_k(\CG)$.
\end{proposition}

\begin{proof}
Let $U$ denote the $p$-smooth locus of $X$. It is a union of strata by
our assumption (S2). Because we have assumed that $X$ is irreducible
there exists a unique open dense stratum $X_{\lambda} \subset X$. Let
$\CP$ be the corresponding indecomposable parity sheaf normalised so
that its restriction to $X_{\lambda}$ is $\ul{k}_{X_{\lambda}}$.

In the following it will be useful to work with non-equivariant
sheaves. Note that the non-equivariant analogue of Theorem
\ref{thm-parityunique} is valid (see \cite[Theorem 2.12]{JMW2}) and
$\overline{\CP} := \For(\CP)$ is the indecomposable non-equivariant
parity sheaf with support $X$.

Let $U'$ denote the largest open union of strata $X_{\lambda}$ for
which $\overline{\CP}_{X_{\lambda}} \cong \ul{k}_{X_{\lambda}}$. We claim
$U = U'$.

Indeed, if $U'$ denotes this set then $\overline{\CP}_{U'}$ satisfies
the properties (1) and (2) of the IC-complexes and hence also
satisfies (3) because $\DD(\overline{\CP}_{U'}) \cong
\overline{\CP}_{U'}[2 \dim X]$. Hence $\overline{\CP}_{U'} \cong
\IC(U,k)$ and so $U' \subset U$ by Propositions
\ref{prop:psmoothconstant} and \ref{prop:icstalks+psmooth}. On the
other hand, $\IC(U,k) \cong \ul{k}_U$ is certainly indecomposable and
$*$-even. It is even $!$-even because $\DD \IC(U,k) \cong \IC(U,k)[2
\dim X]$. Hence $\CP_U \cong \ul{k}_U$ by the classification of parity
sheaves, together with the fact that the restriction of an
indecomposable parity sheaf to an open union of strata is either zero
or indecomposable (see \cite[Proposition 2.11]{JMW2}).

Now, by Theorem \ref{thm:Wparity=BM}, $\DW(\CP) \cong \SB(\lambda)$
and hence $\overline{\CP}_{X_{\mu}} \cong \ul{k}_{X_{\lambda}}$ if and only if
$\SB(\lambda)^{\mu} \cong S_k$. The proposition then follows by
definition of the $k$-smooth locus of the moment graph of $X$.
\excise{ Because we have assumed that $X$ is irreducible there exists
  a unique open dense stratum $X_{\lambda} \subset X$. Let $\CP :=
  \CP(\lambda)$ be the corresponding indecomposable parity sheaf
  normalised so that its restriction to $X_{\lambda}$ is
  $\ul{k}_{X_{\lambda}}$. By Proposition \ref{prop:parityselfdual} we
  have $\DD \CP \cong \CP[2d]$ where $d := \dim X$. Now set
\[
U := \bigsqcup_{\lambda \in \Omega_k(\CG)} X_{\lambda}
\]
and consider $\CP_{U}$, the restriction of $\CP$ to $U$. We certainly have $\DD \CP_U \cong \CP_U[2d]$. By definition of the smooth locus of the moment graph and Theorem \ref{}, we have $\THypgr( \CP_{x_{\lambda}}) \cong S_k$ for all $\lambda \in \Omega_k(\CG)$.  Hence $\CP_{X_{\lambda}} \cong \ul{k}_{X_{\lambda}}$ for all $\lambda \in \Omega_k(\CG)$.}
\end{proof}

Combining this result with Theorem \ref{thm:mgsmoothlocus} yields:

\begin{theorem} \label{thm:psmoothmaintheorem}
A $T$-fixed point $x_{\mu} \in X_{\mu}$ belongs to the $p$-smooth locus of $X$ if and only if for all $\lambda \ge \mu$ the number of one-dimensional $T$-orbits having $x_{\lambda}$ in their closure is equal to the complex dimension of $X$. Moreover, $X_{\mu}$ belongs to the $k$-smooth locus if and only if its $T$-fixed point $x_{\mu}$ does.
\end{theorem}

\subsection{A freeness result}\label{subsec:freeness}

In this subsection $k$ denotes a complete local ring and $p$ denotes
the characteristic of the residue field of $k$.  Let
$X=\bigsqcup_{\lambda\in\Lambda} X_\lambda$ be an irreducible,
complex, stratified $T$-variety. Assume that 
these data satisfy the assumptions (A1)--(A4a/b), (S1), (S2). We
further assume that there exists an indecomposable parity sheaf $\CP$
corresponding to the unique open stratum $X_{\mu} \subset X$. (For
example, $X$ could be open in a stratified variety
satisfying (R1), (R2) and (R3).) For any $\lambda \in \Lambda$ let
\begin{gather*}
X_{>\lambda} = \bigsqcup_{\gamma > \lambda} X_{\gamma} \quad \text{and} \quad
X_{\ge \lambda} = \bigsqcup_{\gamma \ge \lambda} X_{\gamma}
\end{gather*}

For any $\lambda \in \Lambda$ we can find a
$T$-stable affine neighbourhood $U$ of $x_\lambda$ and a $T$-invariant
affine normal slice $N \subset U$ to the stratum $X_{\lambda}$. The aim of this section
is to show the following result:

\begin{proposition} \label{prop:free}
$\THypgr(\CP_{N \setminus \{ x_\lambda \} })$ is torsion free over $k$.
\end{proposition}

Of course this result has no content if $k$ is a
field. However it seems to be quite useful if $k$ is, for example,
the $p$-adic integers. Before turning to the proof of this result we state a corollary, which
is of central importance to \cite{JW}:

\begin{corollary}
If $ X_{>\lambda} $ is $p$-smooth then $H_T^{\bullet}(N\setminus \{ x_\lambda \}, k)$ is a free
  $k$-module.
\end{corollary}

\begin{proof}
  If $X_{> \lambda }$ is $p$-smooth then the constant sheaf with
  coefficients in $k$ is self-dual and hence parity. Hence the restriction of $\CP$ to
  $X_{> \lambda}$ is isomorphic to the constant sheaf (cf.  \cite[Proposition 2.11]{JMW2}). The result then
  follows from Proposition \ref{prop:free}.
\end{proof}

\begin{proof}[Proof of Proposition \ref{prop:free}]
Consider the Cartesian diagram:
\[
\xymatrix{
N \ar[r]& U \\
\{ x_\lambda \} \ar[u]^{{i}}  \ar[r]& X_{\lambda} \ar[u]^{\tilde
  i}
}
\]
Without loss of generality we may assume that $X_{\lambda}$ is a
closed stratum in $X$.  In this case we have seen in the course of the proof of Theorem
\ref{thm:fullyfaithful} that we have an
isomorphism $\THypgr(\tilde{i}^! \CP) \cong
\mathbb{W}(\CP)_\lambda$. Moreover, because $N \hookrightarrow U$ is a
normally non-singular inclusion we have
\[
\THypgr(i^! \CP_N) = \THypgr(\tilde{i}^! \CP) \cong
\mathbb{W}(\CP)_\lambda.
\]
One the other hand, by the attractive proposition, we have $\THypgr(\CP_N) =
\THypgr(\CP_{x_{\lambda}}) = \mathbb{W}(\CP)^\lambda$.

Now consider the open-closed decomposition:
\[
\{x_\lambda \} \stackrel{i}{\hookrightarrow} N
\stackrel{j}{\hookleftarrow} N \setminus \{ x_{\lambda} \} .
\]
This leads to a distinguished triangle
\[
i_!i^! \CP \to \CP_U \to j_*j^* \CP \stackrel{[1]}{\to}.
\]
Taking hypercohomology and using the above observations we conclude
that we have an exact sequence
\[
0 \to \mathbb{W}(\CP)_\lambda \to \mathbb{W}(\CP)^\lambda \to \THypgr(\CP_{N
  \setminus \{ x_{\lambda} \} }) \to 0
\]
where the first map is the inclusion. It follows that we have an embedding
\[
\THypgr(\CP_{N\setminus \{x_\lambda\} })  \hookrightarrow \bigoplus_{E : \lambda \to
  \gamma}  \mathbb{W}(\CP)_E
\]
Now each $\mathbb{W}(\CP)_E$ is isomorphic to a direct sum of shifts of $S/(\alpha_E)$.
By assumption (A4b) no character $\alpha_E$ is $p$-divisible in $X(T)$ and
hence each $S/(\alpha_E)$ is torsion free over $k$. It follows that
$\THypgr(\CP_{N\setminus \{x_\lambda \} })$ is torsion free over $k$. \end{proof}

\section{Representations of reductive algebraic groups} Let $G$ be a simple reductive algebraic group over $\overline{\DF}_p$ and let $\Rep G$ denote the category of rational representations of $G$. It is a fundamental problem in representation theory to determine the characters of the simple and tilting modules in $\Rep G$. 
For simple modules there exists a conjecture, due to Lusztig, in the
case that the characteristic $p$ is larger than the Coxeter number $h$
associated to $G$. For tilting modules there is no general
conjecture. Schur-Weyl duality can be used to show that knowledge of the characters of tilting
modules for $G = GL_n(\overline{\DF}_p)$ implies dimension formula for
the simple modules for $S_m$ for $m \le n$ in characteristic $p$.

We want to explain how the above results allow one to reinterpret these two basic problems using the geometry of certain Schubert varieties in the (complex) affine Grassmannian associated with the Langlands dual group. To this end let $T \subset B \subset G$ denote a maximal torus and Borel subgroup of $G$ respectively. Let $X^*(T)$ denote the character lattice and $X^+(T)$ denote the subset of dominant weights. Then $X^+(T)$ parametrises both the simple and tilting modules in $\Rep G$.

\subsection{Tilting modules}
Let $G^{\vee}_{\DC}$ denote the complex Langlands dual group of $G$,
$G_{\DC}^{\vee}((t))$ its loop group, $\widehat{T^\vee_{\DC}}=T^\vee_\DC\times\DC^\times$ the extended
torus and $\Gr^\vee :=G^{\vee}_{\DC}((t)) / G^{\vee}_{\DC}[[t]] $ the
corresponding affine Grassmannian. Then $X^+(T)$ also parametrises the
$G^{\vee}_{\DC}[[t]]$-orbits on $\Gr^\vee$ and $\Gr^\vee$ satisfies our
assumptions when viewed with the action of $\widehat{T_{\DC}^\vee}$
(indeed, the closures of $G^{\vee}_{\DC}[[t]]$-orbits are examples of
Kac-Moody Schubert varieties).  Recall that the geometric Satake equivalence
\cite{MV} establishes a tensor equivalence between the abelian category of rational representations of $G$ and the tensor category of $G^{\vee}_{\DC}[[t]]$-equivariant perverse sheaves on $\Gr^\vee$.

Recall the following two results which are Theorem 5.1 and Corollary
5.8 of \cite{JMW3}:
\begin{enumerate}
\item If  $p > h+1$, then parity sheaves correspond under
  the geometric Satake isomorphism to
  tilting modules. More precisely, the indecomposable parity sheaf
  $\CP(\lambda)$ corresponds to the indecomposable tilting module
  $T(\lambda)$.
\item The rank of $\THypgr(\CP(\lambda)_{\mu})$ is equal to the
  dimension of the $\mu$-weight space of the tilting
  module $T(\lambda)$.
\end{enumerate}
With the above results in mind, it seems natural to expect that the
Braden-MacPherson algorithm can be used to calculate the characters of
tilting modules. There is a problem, however: the moment graph of the
affine Grassmannian satisfies the GKM-condition if and only if $k$ is
of characteristic $0$.

To get around this problem we take $k$ to be the ring $\DZ_p$ of
$p$-adic numbers. For this the GKM-condition is satisfied. Moreover it
is shown \cite[Proposition 2.41]{JMW2} that the graded ranks of the stalks of
parity sheaves depend only on the characterstic of the residue
field. The following theorem then follows from the above discussion
and our main theorem:

\begin{theorem} Suppose that  $p > h +1$ (see \cite{JMW3} for better bounds).
 When conducted with coefficients in the ring $k = \DZ_p$ of $p$-adic numbers, the Braden-MacPhersons algorithm computes the characters of tilting modules. More precisely, for any character $\mu \in X^\ast(T)$, the dimension of the $\mu$-weight space of $T(\lambda)$ is equal to the rank of $\SB(\lambda)^{\mu}$.
\end{theorem}

\subsection{Simple rational characters}
We now  turn to the application of the above results to Lusztig's
conjecture. Let $I^\vee \subset
G^{\vee}_{\DC}((t))$ be the Iwahori subgroup containing $B^\vee$ and let $\Fl^\vee :=
G^{\vee}_{\DC}((t)) / I^\vee$ denote the affine flag variety with its
$\widehat{T_{\DC}^\vee}$-action.

In \cite{Fie2} a certain subcategory $\CI\subset D^b_{\widehat
  T}(\Fl^\vee,k)$ of {\em special equivariant sheaves} was considered and a
functor $\Phi\colon \CI\to \CR$ was defined, where $\CR$ is a category
of projective objects in a category $ \CC$  naturally associated to
the Lie algebra of $G$. For the application to Lusztig's conjecture one needs to consider only objects in  $\CI$ that are
supported on a  certain Schubert variety $X_{res} \subset \Fl^\vee$. 

An intermediate step in the construction of $\Phi$ was a functor from $\CI$ to the category of
Braden--MacPherson sheaves on the moment graph associated to $\Fl^\vee$. It
turns out that $\CI$ is the category of parity sheaves on $\Fl^\vee$ (with
respect to the stratification by Schubert cells). Indeed, the category $\CI$ is generated from the skyscraper sheaf on the point stratum on $\Fl^\vee$ by repeatedly applying the functors $\pi_s^\ast\pi_{s\ast}$ for simple reflections $s$, where  $\pi_s\colon \Fl^\vee\to \Fl^\vee_s$ is the projection onto the partial affine flag variety associated to $s$. Now parity sheaves are preserved by these functors (cf. \cite[Proposition 4.9]{JMW2}). From the results in \cite{Fie2} we can hence deduce that the ranks of the stalks of parity sheaves determine baby Verma multiplicities for projective objects in $\CC$. These multiplicities in turn determine the characters of simple rational representations of $G$. 
Using the results of this paper we can rephrase the above as follows. Given $\lambda \in \Lambda$ let $\IC({\overline X_{\lambda}}, \DZ)$ denote the intersection cohomology complex of ${\overline X_{\lambda}}$ with integral coefficients (cf. \cite{decperv}).

\begin{theorem}
Suppose that the stalks and costalks of the intersection cohomology complex $\IC(\overline{X_{\lambda}}, \DZ)$ are $p$-torsion-free for all strata $X_{\lambda} \subset X_{res}$. Then the characters of the simple modules for $G$ are given by Lusztig's conjecture.
\end{theorem}

\begin{proof} 
It is known (cf. \cite{Lu2}) that $\IC({\overline X_{\lambda}}, \DZ)
\stackrel{L}{\otimes} \DQ$ is isomorphic to the parity sheaf
$\CP_\DQ(\lambda)$ with coefficients in $\DQ$. Hence $\IC({\overline X_{\lambda}}, \DZ)
\stackrel{L}{\otimes} \DF_p$ is isomorphic to the parity sheaf
$\CP_{\DF_p}(\lambda)$ with coefficients in $\mathbb{F}_p$ if and only if the conditions
of the theorem are met. The theorem then follows  from 
\cite{Fie3, Fie2}  together with our main theorem.
\end{proof}
\def\cprime{$'$} \def\cprime{$'$}

\end{document}